\documentclass{article}

\usepackage{a4wide}
\usepackage{amsfonts}
\usepackage{amsmath}
\usepackage{amssymb}
\usepackage{amsthm}
\usepackage{algorithm,algpseudocode}
\usepackage{pgfplots}
\pgfplotsset{compat=1.8}
\usepackage{graphicx}
\usepackage{color}
\usepackage{colortbl}   
\usepackage{multirow}
\usepackage{boxedminipage}
\usepackage[caption=false]{subfig}
\usepackage{textcomp}
\usepackage{MnSymbol} 
\usepackage[breaklinks=true]{hyperref}
\usepackage{placeins}
\usepackage{stackengine}

\usepackage[capitalize,noabbrev]{cleveref}
\crefname{equation}{}{}
\Crefname{equation}{Equation}{Equation}

\numberwithin{algorithm}{section}
\numberwithin{equation}{section}
\numberwithin{figure}{section}

\newcommand{\N}{\ensuremath{\mathbb{N}}}

\newcommand{\T}{\ensuremath{\mathbb{T}}}

\newcommand{\Z}{\ensuremath{\mathbb{Z}}}

\newcommand{\C}{\ensuremath{\mathbb{C}}}

\newcommand{\ii}{\textnormal{i}}
\newcommand{\e}{\textnormal{e}}
\newcommand{\norm}[1]{\left\Vert #1\right\Vert}

\newcommand{\ceil}[1]{\left\lceil#1\right\rceil}
\newcommand{\floor}[1]{\left\lfloor#1\right\rfloor}

\renewcommand{\ln}{\mathrm{ln\,}}

\newcommand{\boldx}{{\ensuremath{\boldsymbol{x}}}}

\newcommand{\boldk}{{\ensuremath{\boldsymbol{k}}}}
\newcommand{\boldh}{{\ensuremath{\boldsymbol{h}}}}
\newcommand{\boldl}{{\ensuremath{\boldsymbol{l}}}}
\newcommand{\boldz}{{\ensuremath{\boldsymbol{z}}}}

\newcommand{\boldzero}{{\ensuremath{\boldsymbol{0}}}}

\definecolor{darkgreen}{rgb}{0.0,0.5,0.0}
\definecolor{darkorange}{RGB}{255,90,0}

\DeclareMathOperator{\sgn}{sgn}

\newtheorem{theorem}{Theorem}[section]
\newtheorem{lemma}[theorem]{Lemma}
\newtheorem{remark}[theorem]{Remark}
\newtheorem{definition}[theorem]{Definition}
\newtheorem{example}[theorem]{Example}
\newtheorem{corollary}[theorem]{Corollary}

\makeatletter
\def\imod#1{\allowbreak\mkern10mu({\operator@font mod}\,\,#1)}
\makeatother

\pgfdeclareplotmark{hexagon}
{\pgfpathmoveto{\pgfqpoint{0pt}{\pgfplotmarksize}}
  \pgfpathlineto{\pgfqpointpolar{150}{\pgfplotmarksize}}
  \pgfpathlineto{\pgfqpointpolar{210}{\pgfplotmarksize}}
  \pgfpathlineto{\pgfqpointpolar{270}{\pgfplotmarksize}}
  \pgfpathlineto{\pgfqpointpolar{330}{\pgfplotmarksize}}
  \pgfpathlineto{\pgfqpointpolar{30}{\pgfplotmarksize}}
  \pgfpathclose
\pgfusepathqstroke
}

\pgfdeclareplotmark{hexagon*}
{\pgfpathmoveto{\pgfqpoint{0pt}{\pgfplotmarksize}}
  \pgfpathlineto{\pgfqpointpolar{150}{\pgfplotmarksize}}
  \pgfpathlineto{\pgfqpointpolar{210}{\pgfplotmarksize}}
  \pgfpathlineto{\pgfqpointpolar{270}{\pgfplotmarksize}}
  \pgfpathlineto{\pgfqpointpolar{330}{\pgfplotmarksize}}
  \pgfpathlineto{\pgfqpointpolar{30}{\pgfplotmarksize}}
  \pgfpathclose
\pgfusepathqfillstroke
}

\newcommand{\sspan}{\textnormal{span}}

\newcommand{\OO}[1]{\mathcal{O}\left(#1\right)}

\renewcommand{\mathbf}[1]{\ensuremath{\boldsymbol{#1}}}
\renewcommand{\textbf}[1]{{\ensuremath{\boldsymbol{#1}}}}

\title{A Deterministic Algorithm for Constructing Multiple Rank\mbox{-}1 Lattices of Near-Optimal Size\thanks{
		\textbf{Funding}: M.\ Iwen and C.\ Gross were supported in part by NSF DMS 1912706.
		L.\ K\"ammerer was supported in part by the DFG 380648269.
		T.\ Volkmer was supported in part by the SAB 100378180.
	}
}
\author{Craig Gross\thanks{Department of Mathematics, Michigan State University, East Lansing, MI 48824, USA (\href{mailto:grosscra@msu.edu}{\mbox{grosscra@msu.edu}}).}
	\and Mark A.\ Iwen\thanks{Department of Mathematics and Department of CMSE, Michigan State University, East Lansing, MI 48824, USA (\href{mailto:iwenmark@msu.edu}{iwenmark@msu.edu}).}
	\and Lutz K\"ammerer\thanks{Faculty of Mathematics, Chemnitz University of Technology, 09107 Chemnitz, Germany (\href{mailto:lutz.kaemmerer@mathematik.tu-chemnitz.de}{lutz.kaemmerer@mathematik.tu-chemnitz.de}, \href{mailto:toni.volkmer@mathematik.tu-chemnitz.de}{toni.volkmer@mathematik.tu-chemnitz.de}).}
\and Toni Volkmer\footnotemark[4]
        }
        
        \hypersetup{
          pdfauthor={Craig Gross and Mark A. Iwen and Lutz K\"{a}mmerer and Toni Volkmer},
          pdftitle={A Deterministic Algorithm for Constructing Multiple rank-1 Lattices of Near-Optimal Size},
          pdfsubject={65T40, 42B05, 68Q17, 68Q25, 42B35, 65T50, 65Y20, 65D30, 65D32},
          pdfkeywords={approximation of multivariate periodic functions, trigonometric polynomials, lattice rule, multiple rank-1 lattice, fast Fourier transform},
          pdfcreator = {pdflatex},
          plainpages=false,
          pdfstartview=Fit,       
          pdfpagemode=UseOutlines, 
          bookmarksnumbered=true, 
          bookmarksopen=false,     
          bookmarksopenlevel=0,   
          colorlinks=true,       
          linkcolor=black,
          citecolor=black,
          urlcolor=black}

\renewenvironment{thebibliography}[1]{
  \begin{oldthebibliography}{#1}
    \setlength{\itemsep}{4.5pt}
    \setlength{\parskip}{0em}
}
{
  \end{oldthebibliography}
}

\begin{document}

\maketitle

\begin{abstract}
	In this paper we present the first known deterministic algorithm for the construction of multiple rank\mbox{-}1 lattices for the approximation of periodic functions of many variables.  The algorithm works by converting a potentially large reconstructing single rank\mbox{-}1 lattice for some $ d $-dimensional frequency set $ I \subset [N]^d $ into a collection of much smaller rank\mbox{-}1 lattices which allow for accurate and efficient reconstruction of trigonometric polynomials with coefficients in $ I $ (and, therefore, for the approximation of multivariate periodic functions).
	The total number of sampling points in the resulting multiple rank\mbox{-}1 lattices is theoretically shown to be less than $ \OO{ |I| \log^{ 2 }(N |I|) } $ with constants independent of $d$, 
	and by performing one-dimensional fast Fourier transforms on samples of trigonometric polynomials with Fourier support in $ I $ at these points, we obtain exact reconstruction of all Fourier coefficients in fewer than $ \OO{d\,|I|\log^4 (N|I|)} $ total operations.

	Additionally, we present a second multiple rank\mbox{-}1 lattice construction algorithm which constructs lattices with even fewer sampling points at the cost of only being able to reconstruct exact trigonometric polynomials rather than having additional theoretical approximation guarantees.
	Both algorithms are tested numerically and surpass the theoretical bounds.
	Notably, we observe that the oversampling factors \#samples$/|I|$ appear to grow only logarithmically in $ |I| $ for the first algorithm and appear near-optimally bounded by four in the second algorithm.


 \end{abstract}

\section{Introduction}

In this paper we consider the efficient numerical approximation of the Fourier series coefficients of a given function $f:[0,1]^d \rightarrow \C$ of $d$ variables.  More specifically, we seek cubature rules for evaluating the integrals
\begin{equation}
\widehat{f_{\boldk}} := \int_{[0,1]^d} f(\boldx) \, \e^{-2\pi\ii\boldk\cdot \boldx} ~\mathrm{d}x
\label{equ:FourierInt}
\end{equation}
for all indices $\boldk \in I \subset \Z^d$, $|I|<\infty$, which will be exact whenever $f$ is itself a multivariate trigonometric polynomial of the form
\begin{equation}
f(\boldx) = \sum_{\boldk \in I} \widehat{f_{\boldk}} \, \e^{2\pi\ii\boldk\cdot \boldx}.
\label{equ:trigPolyf}
\end{equation}
Existing numerical methods based on rank\mbox{-}1 lattice rules \cite{temlyakov1986reconstruction,KaPoVo13,Kae17,kuo2019function,plonka2018numerical} solve this problem by finding a prime number $M \in \Z$ and generating vector of integers $\boldz \in \Z^d$ such that 
\begin{equation}
\widehat{f_{\boldk}} = \frac{1}{M}\sum_{j = 0}^{M-1} f(\boldx_j ) \, \e^{\frac{-2\pi\ii \boldk \cdot \boldx_j}{M}}
= \frac{1}{M} \sum_{j = 0}^{M-1} f\left( \frac{(j \boldz)\bmod{M}}{M}\right) \, \e^{\frac{-2\pi\ii j \boldk\cdot \boldz}{M}}
,\quad \boldk\in I,
\label{equ:Rank1Cubature}
\end{equation}
holds for all $f$ as per \cref{equ:trigPolyf},
where
$\boldx_j$
are the nodes of the rank\mbox{-}1 lattice
\begin{equation*}
\Lambda(\boldz,M):=\left\{\boldx_j:=\frac{(j \boldz)\bmod{M}}{M}\colon j=0,\ldots,M-1\right\}.
\end{equation*}

One particularly attractive attribute of such cubature rules \cref{equ:Rank1Cubature} is that they allow the full set of Fourier series coefficients $\left\{ \widehat{f_{\boldk}} \colon \boldk \in I \right\}$ to be computed efficiently using a single one-dimensional FFT whenever the modulus function $m_{\boldz,M}:  I \rightarrow [M] \subseteq \Z$ given by $m_{\boldz, M}(\boldk) := \boldk\cdot\boldz\bmod{M}$ is injective, where $[M] := \{ 0, \dots, M - 1 \}$.  As a result, fast Component-By-Component (CBC) methods have been developed for constructing $\boldz$ and $M$ which guarantee that $m_{\boldz, M}$ is injective for any given $I \subset \Z^d$, $|I|<\infty$ (see, e.g., \cite{kuo2019function}).  Such CBC algorithms typically compute the image of $m_{\boldz, M}$, $m_{\boldz, M}(I) \subseteq [M] \subset \Z$, as they run. The image can then be stored and used to rapidly evaluate $m^{-1}_{\boldz, M}:  m_{\boldz, M}(I) \rightarrow I$ on demand thereafter.  Thus, rank\mbox{-}1 lattice approaches provide relatively simple and efficient $\mathcal{O}(M \log M)$-time algorithms for high dimensional Fourier approximation by effectively reducing $d$-dimensional Fourier transforms over arbitrary frequency sets $I$ to one-dimensional FFTs.

An added benefit of rank\mbox{-}1 lattice techniques is that they also trivialize the extension of fast compressive sensing \cite{foucart2013book} and best $k$-term approximation \cite{cohen2009compressed} methods developed for Fourier-sparse vector data  \cite{gilbert2014recent,merhi2019new,bittens2019deterministic,iwen2010combinatorial,Iw13,segal2013improved} to the setting of multivariate trigonometric approximation.  Suppose, for example, that the a priori unknown Fourier support set $I' := \left\{ \boldk \in I \colon \widehat{f_{\boldk}} \neq 0 \right\} \subset I$ of $f$ has 
cardinality $k := \left|I' \right| \ll \left| I \right| =: s$ so that the function $f$ in \cref{equ:trigPolyf} is Fourier $k$-sparse in~$I$.  In this case one can simply $(i)$ use a rank\mbox{-}1 lattice approach to reduce the problem of recovering $\left\{ \widehat{f_{\boldk}} \right\}_{\boldk \in I}$ to a one-dimensional sparse Fourier recovery problem aimed at finding the $k$-sparse vector $\boldsymbol{\hat f} \in \C^{|M|}$ with entries given by
\begin{equation*}{\hat f}_{l} := \begin{cases}
\widehat{f_{m^{-1}_{\boldz, M}(l)}} & \text{if } l \in m_{\boldz, M}(I'),\\
0 & \text{else},\\
\end{cases}\end{equation*}
then $(ii)$ compute a $k$-sparse approximation $\boldsymbol{a} \in \C^{M}$ of $\boldsymbol{\hat f}$ using a one-dimensional Sparse Fourier Transform (SFT) method (see e.g.~\cite{gilbert2014recent} for a survey of such methods), and finally $(iii)$ map the recovered one-dimensional $k$-sparse approximation $\boldsymbol{a}$ of $\boldsymbol{\hat f}$ back to an approximation $a:[0,1]^d \rightarrow \C$ of $f$ by setting
\begin{equation*}
a(\boldx) := \sum_{l \in m_{\boldz, M}(I)} a_{m^{-1}_{\boldz, M}(l)} \, \e^{2\pi\ii \left( m^{-1}_{\boldz, M}(l)~\cdot~ \boldx \right)} .
\end{equation*}
Using the fastest available SFT methods \cite{merhi2019new,bittens2019deterministic} one can provably achieve best $k$-term approximation guarantees of periodic functions $f$ in just $\mathcal{O}\left(k \cdot \log^{5.5}(M) \right)$-time using the strategy above provided that $m^{-1}_{\boldz, M}$ has been efficiently encoded in advance for the index set $I$ of interest.  As a result, such rank\mbox{-}1 lattice reduction approaches allow one to improve the polynomial dependence of previously existing SFT methods for functions of many variables \cite{Iw13,morotti2017explicit} on the number of variables $d$ when the index set~$I$ above is chosen to be, e.g., a hyperbolic cross as opposed to a full integer cube. 

\subsection{Results and Motivation}

Despite their many nice properties, rank\mbox{-}1 lattices suffer from the shortcoming that the prime~$M$ above must generally scale like $\Omega\left(|I|^2 \right)$ for an arbitrary index set $I$ when the modulus function $m_{\boldz,M}$ should be injective. 
Unfortunately, this limits the reduction in computational cost that one can achieve by choosing index sets $I \subset [N]^d + \boldh$ with $|I| \ll N^d$ and a fixed arbitrary shift $\boldh\in\Z^d$ in all of the methods described above.  
In this paper we propose a deterministic algorithm to address this issue by transforming a previously computed rank\mbox{-}1 lattice of size $M = \Omega\left(|I|^2 \right)$ into a set of $L = \mathcal{O}(\log |I|)$ smaller rank\mbox{-}1 lattices each of size at most\footnotemark\ $\mathcal{O}\left(|I| \log^2(N |I|) \right)$ which still collectively allow for cubature rules that exactly integrate the Fourier series coefficients in \cref{equ:FourierInt} for all indices $\boldk \in I \subset [N]^d + \boldh$ and multivariate trigonometric polynomials $f$ as per~\cref{equ:trigPolyf}.

More specifically, we provide the first known deterministic algorithm for constructing multiple rank\mbox{-}1 lattices \cite{Lutz18} for any given index set $I \subset [N]^d + \boldh$.  The proposed algorithm takes a given rank\mbox{-}1 lattice generating vector $\boldz \in [M]^d$ for $I$ as input and effectively uses it to generate new cubature formulas satisfying
\begin{equation}\label{equ:MultRank1Cubature}
\widehat{f_{\boldk}} = \frac{1}{\tilde P_{ \nu(\boldk) }} \sum_{j = 0}^{\tilde P_{\nu(\boldk)}-1} f\left( \frac{(j \boldz)\bmod{\tilde P_{\nu(\boldk)}}}{\tilde P_{\nu(\boldk)}} \right) \e^{\frac{-2\pi\ii j \boldk\cdot \boldz}{\tilde P_{\nu(\boldk)}}}
\end{equation}
for all $\boldk \in I$ and $f$ as per \cref{equ:trigPolyf}, where $\nu:  I \rightarrow [L]$ is a function which determines which of the $L$ smaller lattices we provide below should be used to reconstruct each desired Fourier series coefficient.  Note that the same generating vector $\boldz$ is used for each smaller lattice in \cref{equ:MultRank1Cubature} despite the fact that the lattice size $\tilde P_{\nu(\boldk)}$ varies with $\boldk$, as well as that \cref{equ:MultRank1Cubature} can still be computed using a one-dimensional FFT for each of the $L = \mathcal{O}(\log |I|)$ smaller lattice sizes $\tilde P_0, \dots, \tilde P_{L-1}$.  Finally, it is important to emphasize that the total number of function evaluations required by this modified cubature rule \cref{equ:MultRank1Cubature} is only\footnotemark[\value{footnote}] $\mathcal{O}(|I| \log^2(N |I|))$ as opposed to the $\mathcal{O}(|I|^2)$ function evaluations generally required by a single rank\mbox{-}1 lattice approach \cref{equ:Rank1Cubature}.

\footnotetext{These bounds are simplifications of those in \cref{lem:estimate_PqKm1,thm:not_more_than_half_of_frequencies_collide} under the mild assumptions that the dimension $ d $ and size of the original single rank\mbox{-}1 lattice $ M $ are bounded polynomially by $ \max\{|I|, N\} $.
The latter assumption holds for single rank\mbox{-}1 lattices constructed by CBC methods, cf.\ \cref{sub:analysis_of_lattice_construction}.}

\subsubsection{Main Result}

We are now ready to begin constructing the promised spatial discretization of the trigonometric polynomials in the polynomial space $\Pi_I:=\sspan\{\e^{2\pi\ii\boldk\cdot\circ}\colon\boldk\in I\}$, determined by $I\subset\Z^d$, $|I|=s$. 
The crucial assumption we begin with below is that we already know a so-called \emph{reconstructing} single rank\mbox{-}1 lattice $\Lambda(\boldz,M,I)$ in advance, i.e., a rank\mbox{-}1 lattice which fulfills the reconstruction property that the modulus function discussed above $ m_{ \boldz, M }: I \rightarrow [M] $ is injective on the frequency set $I$, which in particular implies \cref{equ:Rank1Cubature}.
See, e.g., \cite{plonka2018numerical,kuo2019function} for additional details and background about such lattices.

With this setup in hand, \cref{sec:main_result} of this paper is devoted to proving this main theorem concerning the proposed cubature rules on multiple rank\mbox{-}1 lattices.
\begin{theorem}
	\label{thm:summary_theorem}
	Let $ I\subset \Z^d$ be some frequency set with cardinality $ |I| = s $
	and expansion $N_I:=\max_{j=1,\ldots,d}\left(\max_{\boldk\in I}k_j-\min_{\boldh\in I} h_j\right)$.
	If $ \Lambda(\boldz, M, I) $ is a reconstructing single rank\mbox{-}1 lattice, then one can deterministically construct multiple rank\mbox{-}1 lattices $ \Lambda(\boldz, \tilde P_0), \ldots, \Lambda(\boldz, \tilde P_{ L - 1 }) $ such that the Fourier coefficients $ \{\widehat{f_{ \boldk }} \colon \boldk \in I\} $ of any trigonometric polynomial $ f \in \Pi_I $ of the form \cref{equ:trigPolyf} can be exactly reconstructed using only samples of $ f $ on these lattices by the cubature rule \cref{equ:MultRank1Cubature}.
	Moreover, the total number of function evaluations on these lattice points is bounded by
	\begin{equation*}
		\sum_{ \ell = 0 }^{ L - 1 } \tilde{P}_\ell \leq 
		\begin{cases}
			2 &\text{for } s = 1, \\
			6\,s\, \log_2(d N_I M) \; \ln\left(3 \frac{s}{\log_2 s} \log_2(d N_I M)\right) & \text{for } s \geq 2.
		\end{cases}
	\end{equation*}
	The total computational complexity for the construction of these rank\mbox{-}1 lattices can be bounded by 
	\[
	\OO{s^2\,\log s\, \log{dN_IM}+\;s\left(d+(\log{dN_IM})\log(\log{dN_IM})\right)},
	\]
	and the total computational complexity for reconstructing the Fourier coefficients can be bounded by 
	\[
	\OO{  s \log s\left(d + (\log s)(\log d N_I M)(\log^2 \log_s d N_I M)\right)  }.
	\]
\end{theorem}
\begin{proof}
	The bounds on the total number of samples from the rank\mbox{-}1 lattices follow from \cref{thm:not_more_than_half_of_frequencies_collide}.
	The computational complexity bound for lattice construction follows from \cref{sub:analysis_of_lattice_construction}, and the bound for reconstructing Fourier coefficients follows from \cite[Algorithm~1]{kaemmerer2019multiple}.
\end{proof}

\subsubsection{An Application to SFTs}  

One additional consequence of our deterministic multiple rank\mbox{-}1 lattice construction approach, beyond reducing the number of function evaluations necessary in order to collectively achieve exact cubature rules~\eqref{equ:MultRank1Cubature} for all trigonometric polynomials of type \eqref{equ:trigPolyf} in a deterministic way, is the ability to trivially parallelize SFT methods for the sparse Fourier approximation of periodic functions of many variables in a modularized fashion.  After deterministically transforming an existing rank\mbox{-}1 lattice into multiple rank\mbox{-}1 lattices satisfying \eqref{equ:MultRank1Cubature} for a large index set $I$ of interest, the SFT of one's choice may then be applied on each of the resulting smaller rank\mbox{-}1 lattices in parallel in order to more rapidly and stably collectively discover a superset of the true Fourier support of any function $f$ with $\textrm{supp}\left(\widehat{f}\right) \subset I$.  Fast secondary Fourier coefficient estimation methods (see, e.g., section 4 of \cite{gilbert2014recent}) can then be used to estimate the Fourier coefficients of the discovered frequencies in order to eliminate false positives. As a result of such approaches, we anticipate that the development of novel parallel SFT methods for the approximation of functions of many variables which are both faster and more numerically stable than SFT methods based on a single rank\mbox{-}1 lattice will result from the new lattice constructions presented herein.

\subsubsection{An Application to the Recovery of More General Functions}  

As previously mentioned, we construct a cubature rule \eqref{equ:MultRank1Cubature} that exactly reconstructs all Fourier coefficients
of multivariate trigonometric polynomials with frequencies in a specific frequency set~$I$ which is assumed to be given.
Of course, one can apply these cubature rules in order to compute approximations of the Fourier coefficients of more general periodic functions. The resulting trigonometric polynomial can be used as an approximant. For specific approximation settings, it is clear that
the worst case error of this approximation is almost as good as the approximation one achieves when approximating the Fourier coefficients using the lattice rule that uses all samples of the reconstructing single rank\mbox{-}1 lattice from which we start the construction of our cubature rules, cf.~\cite{KaVo19} for details. 
From that point of view, the strategy we present in this paper even yields a general approach for significantly reducing the number of sampling values used while only slightly increasing approximation errors.
We refer to Remark~\ref{rem:ApproximationResult} for more details and to the numerical example in section~\ref{sec:numerics:symhceven:approx} that yields
Figure~\ref{fig:G_3:rel_sampl_err_L2} illustrating this assertion.

\section{The Proof of Theorem~\ref{thm:summary_theorem}}
\label{sec:main_result}

We denote the $q$th prime number by $P_q$, $q\in\N$. For technical reasons, we define $P_0:=1$.

\begin{lemma}\label{lem:proof_Ps_1d}
Let $J:=\{k_1,\ldots,k_s\}\subset[\tilde{M}]$, $|J|=s$, for $s,\tilde{M}\in\N$ with $\tilde{M}\ge s\ge1$.
Moreover, we determine $q\in\N$ such that $P_{q-1}<s\le P_q$, and $K:=\max\left(1,2(s-1)\ceil{-1+\log_{P_q}\tilde{M}}\right)$.
Then, there exist prime numbers $\tilde{P}_0, \ldots, \tilde{P}_{L-1}\in \mathcal{P}_s:=\{P_\ell\colon \ell=q,\dots,q+K-1\}$, $L \le \log_2{s}+1$
such that
\begin{equation*}
	J=\bigcup_{\ell=0}^{L-1}\{k\in J\colon k\not\equiv h\imod{\tilde{P}_\ell} \text{ for all } h\in J\setminus\{k\}\}
\end{equation*}
holds.
\end{lemma}

\begin{proof}
We assume $s\ge 2$ and $\tilde{M}>P_q$, otherwise the statement is trivial.

Let $\mathcal{P}_s=\{P_\ell\,\colon\ell=q,\ldots,q+K-1\}$ be the set of the $K$ smallest prime numbers not smaller than $P_q$ and
$
Y_{i,j}:=\{P\in \mathcal{P}_s\,\colon k_i\equiv k_j\imod{P}\}
$
a subset 
which collects all primes $P$ in $\mathcal{P}_s$ where the frequencies $k_i\in J$ and $k_j\in J$
collide modulo $P$.
Since $ |k_i - k_j| $ is divisible by each prime $ P $ in $ Y_{ i, j} $, the Chinese Remainder Theorem implies that $ \prod_{ P \in Y_{ i, j } } P $ divides $ |k_i - k_j| < \tilde M $.
Therefore, we observe
\begin{equation*}
P_q^{|Y_{i,j}|} \le \prod_{P\in Y_{i,j}}P < \tilde{M}
\end{equation*}
for all
$i \neq j\in\{1,\ldots,s\}=:S_0$, i.e., $ k_i \neq k_j $, and this implies $|Y_{i,j}| \le \ceil{-1 + \log_{P_q} \tilde{M}}$.

Moreover, we collect all primes for which $k_i$ collides with any other $k_j$ in the sets
\begin{equation*}
Y_i:=\{P\in \mathcal{P}_s\,\colon k_i\equiv k_j\imod{P} \text{ for at least  one } k_j\in J\setminus\{k_i\}\}=\bigcup_{k_j\in J\setminus\{k_i\}}Y_{i,j}.
\end{equation*}
The cardinality of each $Y_i$ is bounded by
\begin{equation*}|Y_i| \le \sum_{k_j\in J\setminus\{k_i\}}|Y_{i,j}| \le (s-1)\ceil{-1+\log_{P_q}\tilde{M}}.\end{equation*}
Accordingly, we count
\begin{equation*}|\mathcal{P}_s\setminus Y_i|= |\mathcal{P}_s|-|Y_i|\ge K- (s-1)\ceil{-1+\log_{P_q}\tilde{M}} \ge|\mathcal{P}_s|/2.\end{equation*}

We define the indicator variables
\begin{equation*}
Z_{i,\ell}:=\begin{cases}
1 & P_\ell \in \mathcal{P}_s\setminus Y_i,\\
0 & P_\ell\in Y_i,
\end{cases}
\end{equation*}
for all $k_i\in J$ and $P_\ell\in \mathcal{P}_s$. Summing up these indicator variables $Z_{i,\ell}$ and using the estimates from above yields
\begin{equation}
\sum_{i\in S_0}\sum_{\ell=q}^{q+K-1} Z_{i,\ell} 
=\sum_{i\in S_0} |\mathcal{P}_s\setminus Y_i| \ge |S_0||\mathcal{P}_s|/2=s|\mathcal{P}_s|/2.
\label{eq:sum_Z}
\end{equation}

We will now show that $\sum_{i\in S_0} Z_{i,\ell} \ge s/2$  holds for at least one $P_\ell\in \mathcal{P}_s$ by contradiction.
To this end, suppose that
$
\sum_{i\in S_0} Z_{i,\ell} < s/2
$
for all $P_\ell\in \mathcal{P}_s$. Accordingly, we estimate
\begin{equation*}
\sum_{\ell=q}^{q+K-1}\sum_{i\in S_0} Z_{i,\ell}< |S_0||\mathcal{P}_s|/2=s|\mathcal{P}_s|/2
\end{equation*}
which is in contradiction to \cref{eq:sum_Z}.
Thus, there exists at least one prime $P_{\ell_0}\in \mathcal{P}_s$ such that
\begin{equation*}\sum_{i\in S_0} Z_{i,\ell_0}= |\underbrace{\{k_i\in J\,\colon k_i\not\equiv k_j\imod{P_{\ell_0}} \text{ for all } k_j\in J\setminus\{k_i\}\}}_{=:J_1}| \ge s/2.\end{equation*}
We set $\tilde{P}_0:=P_{\ell_0}$, and then apply the strategy iteratively.

For $r\in\N$, $r\ge 1$, we define $S_r:=\{i\in S_{r-1}\colon\exists k_j\in J\setminus\{k_i\}\text{ with } k_i\equiv k_j\imod{\tilde{P}_{r-1}}\}$ and obtain $s_r:=|S_r|\le 2^{-r} s$. Obviously, we have
\begin{equation}
J'_r:=\{k_i\colon i\in S_r\}=J\setminus\bigcup_{t=1}^{r}J_r,\label{eq:def_Jprime_r}
\end{equation}
which are the frequencies that collide modulo each of $\tilde{P}_0,\ldots,\tilde{P}_{r-1}$ to some other frequency in~$J$. We reconsider the variables defined above, but now we restrict the indices to $i\in S_r$.
For instance, we observe $\{\tilde{P}_0,\ldots,\tilde{P}_{r-1}\}\subset Y_i$ for all $i\in S_r$.
We estimate
\begin{align*}
\sum_{i\in S_r}\sum_{\ell=q}^{q+K-1}Z_{i,\ell}=\sum_{i\in S_r}|\mathcal{P}_s\setminus Y_i|\ge s_r|\mathcal{P}_s|/2.
\end{align*}
Using the same contradiction as above, we observe that for at least one $P_{\ell_r}\in\mathcal{P}_s\setminus\{\tilde{P}_0,\ldots,\tilde{P}_{r-1}\}$ we have
\begin{equation*}
\sum_{i\in S_r}Z_{i,\ell_r}=|\underbrace{\{k_i\in J'_r\colon k_i\not\equiv k_j\imod{P_{\ell_r}} \text{ for all } k_j\in J\setminus\{k_i\}\}}_{=:J_{r+1}}|\ge s_r/2.
\end{equation*}

We now set $\tilde{P}_r:=P_{\ell_r}$ and increase $r$ up to the point where $0=|S_{r+1}|=s_{r+1}$ holds.
In order to estimate the largest possible step number $ r_{ \text{max} } \geq r $, we require that $ s_{ r_{ \text{max} } + 1 } \le 2^{ -(r_{ \text{max} } + 1) }s < 1 $.
This is satisfied in particular when $ r_{ \text{max} } = \floor{\log_2(s)} $, and thus we bound the total number of primes as $ L \leq r_{ \text{max} } + 1 \leq \log_2(s) + 1 $.
\end{proof}

\begin{remark}
In the proof of \cref{lem:proof_Ps_1d} we determined that there exist primes in the candidate set~$\mathcal{P}_s$ fulfilling the assertion. This set contains the first $K:=\max \left(1, -1 + 2(s-1)\ceil{\log_{P_q}\tilde{M}} \right)$ prime numbers not smaller than $P_q$, $P_{q-1}<s\le P_q$, which only depends on $s$. However, from a theoretical point of view, any prime  number $P$ larger than $\ceil{s/2}$ may fulfill $|J_1|\ge s/2$. Thus, one also could start the set of prime candidates at that point, which would result in a slightly increased cardinality of the candidate set, due to the fact that $K$ depends on the logarithm to the base of the smallest prime in the candidate set. In spite of that increased cardinality, the maximal prime number in the candidate set $P_{q+K-1}$, which is estimated in the next lemma, may be decreased.
Analyzing this approach leads to similar statements as in the previous and the following lemmas with slightly changed constants. In more detail, both constants $C_1$ and $C_2$ can be bounded less than~$3$. However, the proof requires more effort and we could not bound the resulting constants lower than those stated in \cref{lem:estimate_PqKm1}.
\end{remark}

\begin{lemma}\label{lem:estimate_PqKm1}
Assume $s,\tilde{M}\in\N$, $s\le\tilde{M}$, $P_q$ is the smallest prime not smaller than~$s$,
and let $K:=\max\left(1,2(s-1)\ceil{-1 + \log_{P_q}\tilde{M}}\right)$.
Then, we estimate 
\begin{equation*}
P_{q+K-1}\le \begin{cases}
2 & \text{for } s=1,\\
C_1 s \, (\log_s\tilde{M}) \, \ln (C_2s\log_s\tilde{M}) & \text{for } s\ge2,\\
\end{cases}
\end{equation*}
with absolute constants $ C_1 < 2.3 \, (1 + \e^{ -3/2 }) \le 2.832 $ and $ C_2 \le 2.3 $.
\end{lemma}

\begin{proof}
For $s=1$, we observe $P_{q+K-1}=P_q=2$.\newline
When $s\ge2$ and $P_q\ge\tilde{M}$ we have $K=1$ and $P_q< 2s$ as a result of Bertrand's postulate.

We then consider $s\ge 2$ and $P_q< \tilde{M}$ which yields
\begin{equation*}q+K-1=q-1+2(s-1)\ceil{-1 + \log_{P_q}\tilde{M}}\le q-1+2(s-1)\log_{P_q}\tilde{M}.\end{equation*}

We distinguish two cases, where the final constants from the lemma are determined by the second case.
In the first, we restrict to the finite range where $2\le s \le 8$ with $P_q < \tilde{M} < P_q^{ \left \lceil 10 / (s - 1) \right \rceil  }$, and numerically check that the upper bound
\begin{equation*}
	P_{ q + K - 1 } < 2.831 \, s \, \log_s \tilde M \ln \left( 2.3 \, s \, \log_s \tilde M \right)
\end{equation*}
is satisfied.
In the second case, where $2 \leq s \leq 8$ with $ \tilde M \geq P_q^{ \left \lceil 10/(s - 1) \right \rceil  } $ or $ s \geq 9 $, we have $ q - 1 + K \geq 20 $.
We then estimate this quantity from above as
\begin{align*}
	q + K - 1  
		&\leq q - 1 + 2(s - 1) \log_s \tilde M = \left( \frac{q - 1}{s \log_s \tilde M} + 2 \frac{s - 1}{s}  \right)s \log_s \tilde M \\
		&\leq \left( \frac{q - 1}{s} + 2 \frac{s - 1}{s}  \right) s \log_s \tilde M \leq 2.3 \,s \,\log_s \tilde M
\end{align*}
where one achieves the last estimate by computing $ \frac{q - 1}{s} + 2 \frac{s - 1}{s} $ for $ 2 \leq s < 66 $ and for $ s \geq 66 $, one obtains
\begin{equation*}\frac{q-1}{s}+2\frac{s-1}{s}\overset{\text{\cite[Eq.\,(3.6)]{RoSchoe62}}}{\le} \frac{1.25506}{\ln s}+2\le \frac{1.25506}{\ln 66}+2<2.3\,.\end{equation*}
By the estimate
\begin{equation*}
	e^{ -1/2 }x \ln(x) \leq x^{ 1 + \e^{ -3/2 } } \implies \ln(e^{ -1/2 } x \ln x) = \ln(x) + \ln \ln(x) - \frac{1}{2} \leq (1 + \e^{ -3/2 }) \ln x
\end{equation*}
for $ x > 1 $, an application of \cite[Eq.\,(3.11)]{RoSchoe62} gives
\begin{align*}
P_{q+K-1}&< (q+K-1) \, \big(\ln(q+K-1)+\ln\ln(q+K-1)-1/2\big)\\
&\le (1+\e^{-3/2}) \, (q+K-1) \, \ln(q+K-1)\\
&\le (1 + \e^{ -3/2 }) \, 2.3 \, s \, (\log_s\tilde{M}) \, \ln(2.3 \, s\log_s\tilde{M}),
\end{align*}
as desired.
\end{proof}

\Cref{lem:proof_Ps_1d} ensures the existence of a set of primes $ \tilde{P}_0,\ldots,\tilde{P}_{L-1} $ such that each single element of a given set of integers will not collide modulo at least one $\tilde{P}_\ell$ with any other of these integers. We can now use these primes to convert the large reconstructing single rank\mbox{-}1 lattice $ \Lambda(z, M, I) $ for some frequency set $I$ into smaller rank\mbox{-}1 lattices which, based on their ability to avoid collisions in the frequency domain, will provide a sampling set to exactly reconstruct the Fourier coefficients of all multivariate trigonometric polynomials in $\Pi_I$.

\begin{theorem}
\label{thm:not_more_than_half_of_frequencies_collide}
Let $I\subset\Z^d$, $|I|=s\ge 2$, and a generating vector $\boldz\in[M]^d$ of a reconstructing single rank\mbox{-}1 lattice $\Lambda(\boldz,M, I)$ be given.
We determine 
$\tilde{M}:=\max\{\boldk\cdot\boldz\colon\boldk\in I\}-\min\{\boldk\cdot\boldz\colon\boldk\in I\}+1$.
Then there exists a set of prime numbers $\tilde{P}_0,\ldots,\tilde{P}_{L-1}$, $L\le \log_2{s}+1$, such that
\begin{equation}
I=\bigcup_{\ell=0}^{L-1}\{\boldk\in I\colon \boldk\cdot\boldz\not\equiv \boldh\cdot\boldz\imod{\tilde{P}_\ell} \text{ for all } \boldh\in I\setminus\{\boldk\}\},\label{eq:union_to_I}
\end{equation}
which means that the Fourier coefficients of a multivariate trigonometric polynomial in $\Pi_I$ can be uniquely reconstructed from the sampling values of the rank\mbox{-}1 lattices $\Lambda(\boldz,\tilde{P}_0)$, \ldots, $\Lambda(\boldz,\tilde{P}_{L-1})$.
The number of sampling values used can be bounded by
\begin{equation}
\sum_{\ell=0}^{L-1}\tilde{P}_{\ell}\le 2 \, C_1 s \, (\log_2\tilde{M}) \, \ln (C_2 \, s\log_s\tilde{M}),\label{eq:details_sampling_nodes}
\end{equation}
with constants $C_1$, $C_2$ from \cref{lem:estimate_PqKm1}.
\end{theorem}
\begin{proof}
We consider a multivariate trigonometric polynomial $f\in\Pi_I$ as per~\cref{equ:trigPolyf}, $I\subset\Z^d$, $|I|=s<\infty$,
and determine $\boldk^*\in I$ such that $\boldk^*\cdot\boldz=\min\{\boldk\cdot\boldz\colon\boldk\in I\}$,
which is unique since $\Lambda(\boldz,M,I)$ is a reconstructing single rank\mbox{-}1 lattice for $I$.

For the given vector $\boldz\in\Z^d$, we define the univariate trigonometric polynomial
\begin{equation*}
f^{1\text{d}}(t):=\e^{-2\pi\ii\boldk^*\cdot\boldz\, t}f(t\boldz)
=
\sum_{\boldk\in I}\widehat{f_{\boldk}} \, \e^{2\pi\ii(\boldk\cdot\boldz-\boldk^*\cdot\boldz)\, t} =
\sum_{\boldk\in I}\widehat{f^{1\text{d}}_{l_\boldk}} \, \e^{2\pi\ii l_\boldk\, t},
\end{equation*}
which is $1$-periodic and represents the evaluation of the multivariate trigonometric polynomial~$f$
along the direction given by the vector $\boldz$ times some frequency shift factor.
The (one-dimensional) frequency set of $f^{1\text{d}}$ is then determined by $I^{1\text{d}}:=\{l_\boldk:=(\boldk-\boldk^*)\cdot\boldz\,\colon\boldk\in I\}\subset[\tilde{M}]$
and the mapping $\boldk\mapsto l_\boldk$, $\boldk\in I$, is injective, i.e.,  $|I^{1\text{d}}|=|I|=s$, since $ \Lambda(\boldz, M, I) $ is a reconstructing single rank\mbox{-}1 lattice.
Applying \cref{lem:proof_Ps_1d} with $J=I^{1\text{d}}$ and $\tilde{M}$ as above
we find a set of prime numbers $\{\tilde{P}_0, \ldots, \tilde{P}_{L-1}\}$ with $\tilde{P}_\ell\le P_{q-1+K}$ 
such that \cref{eq:union_to_I} holds.
Thus, sampling at all nodes of the union of $L$ different equidistant sampling schemes
\begin{equation*}
\bigcup_{\ell=0}^{L-1}\left\{0,\frac{1}{\tilde{P}_\ell},\ldots,\frac{\tilde{P}_\ell-1}{\tilde{P}_\ell}\right\},
\end{equation*}
which contains at most $1-L+\sum_{\ell=0}^{L-1}\tilde{P}_\ell$ sampling nodes, will allow for the unique reconstruction of
all Fourier coefficients of $f^{1\text{d}}$, i.e.,
all Fourier coefficients of the multivariate trigonometric polynomial $f$ can be uniquely reconstructed using the sampling values 
of $f$ at
\begin{equation*}
\mathcal{X}:=\bigcup_{\ell=0}^{L-1}\left\{\boldzero,\frac{1}{\tilde{P}_\ell}\boldz,\ldots,\frac{\tilde{P}_\ell-1}{\tilde{P}_\ell}\boldz\right\}
\end{equation*}
using the inversion of the injective mapping $\boldk\mapsto l_\boldk$, $\boldk\in I$.
Due to the periodicity of~$f$, we have
\begin{equation*}
\{f(\boldx)\colon\boldx\in\mathcal{X}\}=
\left\{f(\boldx)\colon\boldx\in \Lambda(\boldz,\tilde{P}_0)\cup\cdots\cup\Lambda(\boldz,\tilde{P}_{L-1})\right\}
\end{equation*}
which yields the first assertion.

Finally, we estimate
\begin{align*}
\left|\bigcup_{\ell=0}^{L-1} \Lambda(\boldz,\tilde{P}_\ell)\right|&\le 
\sum_{\ell=0}^{L-1}\tilde{P}_\ell
\le (\log_2{s}+1) P_{q-1+K}\overset{\text{Lem.~\ref{lem:estimate_PqKm1}}}{\le}
2C_1s(\log_2\tilde{M})\ln (C_2s\log_s\tilde{M}).
\end{align*}
\end{proof}

\begin{remark}
In fact, the one-dimensional frequency shift in the proof of the last theorem can be omitted. 
Then, for $ I_\ell $ the collection of frequencies which do not collide modulo $ \tilde P_\ell $ in $ I $, sampling along the rank\mbox{-}1 lattice $ \Lambda(\boldz, \tilde P_\ell) $ results in an equispaced sampling of the multivariate trigonometric polynomial $f$ along the generating vector~$\boldz$. Applying a one-dimensional DFT in order to determine the Fourier coefficients from these samples yields
\begin{align}
	\label{eq:LatticeDFT}
\widehat{f_\boldk}^{\Lambda(\boldz,\tilde P_\ell)}:=
\frac{1}{\tilde P_\ell}\sum_{j=0}^{\tilde P_\ell-1}f\left(\frac{j\boldz}{\tilde P_\ell}\right)\e^{ \frac{-2\pi\ii j \boldk\cdot\boldz}{\tilde P_\ell}}
=\begin{cases}
\widehat{f_\boldk} & \text{for }\boldk\in I_\ell,\\
\sum_{(\boldk-\boldh)\cdot\boldz\equiv 0\imod{\tilde P_\ell}}\widehat{f_\boldh}
& \text{for }\boldk\in I\setminus I_\ell.\\
\end{cases}
\end{align}
\end{remark}

\begin{remark}\label{rem:discuss_tildeM}
We consider two crucial estimates on $\tilde{M}$ in \cref{thm:not_more_than_half_of_frequencies_collide}
\begin{equation}
%\begin{aligned}
\tilde{M}
%&
=1+\max_{\boldk\in I}\left\{\sum_{i=1}^dk_iz_i\right\}+\max_{\boldh\in I}\left\{\sum_{i=1}^d-h_iz_i\right\}\\
%&
\le 
1+\sum_{i=1}^dz_i\left(\max_{\boldk\in I}k_i-\min_{\boldh\in I}h_i\right)
\le d N_I M
%\end{aligned}
\label{eq:estimate1_tildeM}
\end{equation}
\begin{equation}
%\begin{aligned}
\tilde{M}
%&
=1+\max_{\boldk\in I}\left\{\sum_{i=1}^dk_iz_i\right\}-\min_{\boldh\in I}\left\{\sum_{i=1}^dh_iz_i\right\}\\
%&
\le 2\|\boldz\|_\infty\max_{\boldk\in I}\|\boldk\|_1+1\le 2M\max_{\boldk\in I}\|\boldk\|_1
%\end{aligned}
\label{eq:estimate2_tildeM}
\end{equation}
where $N_I:=\max_{j=1,\ldots,d}\left(\max_{\boldk\in I}k_j-\min_{\boldh\in I} h_j\right)$ is the expansion of the frequency set $I$.

The estimate in \cref{eq:estimate1_tildeM} is a rough but universal upper bound on~$\tilde{M}$ that depends on the dimension~$d$. The inequality in~\cref{eq:estimate2_tildeM} provides a dimension independent upper bound on~$\tilde{M}$
in cases where the frequency set $I$ is contained in an $\ell_1$-ball of a specific size~$R$, i.e., $I\subset\{\boldk\in\Z^d\colon\|\boldk\|_1\le R\}$, which yields $\tilde{M}\le 2MR$.
We refer to \cref{sub:analysis_of_lattice_construction}, where we present and analyze
the computational costs and discuss the advantages of the latter estimate.
\end{remark}

\begin{remark}
	\label{rem:ApproximationResult}
	\sloppy
	By virtue of the Fourier coefficient reconstruction process~\cref{equ:MultRank1Cubature} and the reconstructing property \cref{eq:union_to_I} of the considered multiple rank\mbox{-}1 lattices, theoretical guarantees for approximation with trigonometric polynomials are immediate, see also~\cite{KaVo19}.
	For example, defining the Wiener algebra $ \mathcal{A}(\T^d) := \{ f \in L_1(\T^d) \colon \norm{ f }_{ \mathcal{A}(\T^d) } := \sum_{ \boldk \in \Z^d } |\widehat{f_\boldk}| < \infty \} $, for functions $ f \in \mathcal{A}(\T^d) \cap \mathcal{C}(\T^d) $, each of the DFTs in \cref{eq:LatticeDFT} used to approximate the Fourier coefficients $ \{\widehat{f_\boldk} \colon \boldk \in I\} $ can be shown to produce aliasing errors comparable to the truncation error (see e.g., \cite[Lemma 3.1]{KaVo19}).
	Thus, if we define the truncation $ S_I f := \sum_{ \boldk \in I } \widehat{f_\boldk} \, \e^{ 2\pi \ii \boldk \cdot \circ } $ and the approximation using samples of $ f $ on the generated multiple rank\mbox{-}1 lattices $ S_I^{ \Lambda } f := \sum_{ \boldk \in I } \widehat{f_\boldk}^{ \Lambda(\boldz, \tilde P_{ \nu(\boldk) }) } \e^{ 2\pi \ii \boldk \cdot \circ } $, we have approximation guarantees of the form
	\begin{equation*}
		\| f - S_I^\Lambda f \|_{ L_\infty(\T^d) } \leq \| f - S_I f \|_{ L_\infty(\T^d) } + \| S_I f - S_I^\Lambda f \|_{ L_\infty(\T^d) }\leq (1 + L) \| f - S_I f \|_{ \mathcal{A}(\T^d) }.
	\end{equation*}
	Correspondingly, one can show similar results with respect to the $L_2(\T^d)$ norm, for example, $\norm{ f - S_I^\Lambda f }_{ L_2(\T^d) } \leq (1 + L) \norm{ f - S_If }_{ \mathcal{A}(\T^d) }$, and also error estimates with respect to Sobolev Hilbert spaces of dominating mixed smoothness, cf.~\cite{KaVo19} for more details.
\end{remark}

As considered in \cite[subsection 4.2]{Lutz18} for randomized lattice constructions, we can take an alternative approach to \cref{thm:not_more_than_half_of_frequencies_collide} which requires fewer samples at the cost of having only theoretical reconstruction guarantees for trigonometric polynomials (i.e., the results concerning approximation discussed in \cref{rem:ApproximationResult} do not apply in a straightforward manner).
Rather than require that at each step of the lattice construction, a prime $ P $ is chosen so that a set of frequencies can be obtained which do not collide with \emph{any} other frequency in the original frequency set modulo $ P $, we instead recursively reduce the size of the set that the resulting rank\mbox{-}1 lattice has the reconstruction property over without concern for other frequencies.

\begin{theorem}\label{theorem:lattice_reduction}
Let $I\subset\Z^d$, $|I|=s\ge 1$, $d\ge 2$, $ \tilde M := \max\{\boldk \cdot \boldz \colon \boldk \in I\} - \min\{\boldk \cdot \boldz \colon \boldk \in I|\} + 1  $.
For a reconstructing single rank\mbox{-}1 lattice  $\Lambda(\boldz,M, I)$, there exist primes $\tilde P_0,\ldots,\tilde P_{L-1}$, $L\le\log_2{s} + 1$, with
\begin{equation}
\label{eq:upper_bound_on_sum_of_primes}
\sum_{\ell=0}^{L-1} \tilde P_\ell \le
\begin{cases}
2 & \text{for } s=1,\\
8\,s\,(\log_2 \tilde M)\,\ln( 2 \, \log_2 \tilde M) & \text{for } s\ge2,\\
\end{cases}
\end{equation}
such that for every $ f \in \Pi_I $, the formula
\begin{equation}
\label{equ:MultRank1CubatureReducingFrequencySet}
\begin{split}
\widehat{f_{\boldk}} &= \frac{1}{\tilde P_{ \nu(\boldk) }} \sum_{j = 0}^{\tilde P_{\nu(\boldk)} - 1} f_{\nu(\boldk)-1}\left( \frac{(j \boldz)\bmod{\tilde P_{\nu(\boldk)}}}{\tilde P_{\nu(\boldk)}} \right) \e^{\frac{-2\pi\ii j \boldk\cdot \boldz}{\tilde P_{\nu(\boldk)}}}
\\[1em]
\text{with}\qquad f_{\nu(\boldk)-1}(\boldx)&:=f(\boldx)-\sum_{\boldh\in\{\boldl\colon \nu(\boldl)<\nu(\boldk)\}}\hat{f}_\boldh \, \e^{2\pi\ii \boldh\cdot\boldx}
\end{split}
\end{equation}
holds where $ \nu: I \rightarrow [L]$ maps frequencies to the lattice used to reconstruct the corresponding Fourier coefficient, i.e., we can uniquely reconstruct each multivariate trigonometric polynomial with frequencies in $ I $ using samples along the rank\mbox{-}1 lattices $\Lambda(\boldz,\tilde P_0), \ldots, \Lambda(\boldz,\tilde P_{L-1})$.
\end{theorem}

\begin{proof}
	The proof is simply a recursive application of part of the previously discussed approach, so we only provide a sketch.

	We use only the first prime $ \tilde P_0 $ from \cref{lem:proof_Ps_1d} to determine a set of frequencies $ I_0 \subset I $ such that $ \Lambda(\boldz, \tilde P_0, I_0) $ is a reconstructing single rank\mbox{-}1 lattice with $ |I_0| \geq s / 2 $.
	Performing the reconstruction process in \cref{thm:not_more_than_half_of_frequencies_collide} for only frequencies in $ I_0 $ using samples from $ \Lambda(\boldz, \tilde P_0, I_0) $ recovers the corresponding Fourier coefficients exactly.
	This then defines the correspondence $ \nu(\boldk) = 0 $ for all $ \boldk \in I_0 $.
	Subtracting off the recovered polynomial terms and recursively repeating the process with the frequency set $ I \setminus I_0 $ gives \cref{equ:MultRank1CubatureReducingFrequencySet}.
	
	The upper bound on the number of samples is a result of \cref{lem:estimate_PqKm1}, noting that at each step, the cardinality of the frequency set is reduced by half.
	Splitting the dependence on $ s $ and $ \tilde M $ in the second logarithm using the inequality $ \ln(xy) \leq 2 (\ln x)(\ln y) $ for $ x, y \geq \e $ and estimating the resulting geometric series gives \cref{eq:upper_bound_on_sum_of_primes}.
\end{proof}

\subsection{Analysis of lattice construction}\label{sub:analysis_of_lattice_construction}
\begin{algorithm}[!ht]
	\caption{Deterministic construction of multiple rank\mbox{-}1 lattice suitable for reconstruction and approximation, according to \cref{thm:not_more_than_half_of_frequencies_collide} and \cref{lem:proof_Ps_1d}}
\label{alg:main1}
	\begin{algorithmic}[1]
		\Procedure{\textbf{DeterministicMR1L}}{}\\
		{\bfseries Input:}{ frequency set $I\subset\Z^d$, generating vector $\boldz\in\N_0^d$ of a reconstructing single rank\mbox{-}1 lattice for $I$}\\
		{\bfseries Output:}{ lattice sizes $\tilde{P}_0,\ldots,\tilde{P}_{L-1}$}
		\State\label{alg:main1:innerprod}Compute the set $J'_0:=\{\boldk\cdot\boldz\colon\boldk\in I\}$
		\State Determine $q\in\N$ s.t.\ $P_{q-1}< |I|\le P_q$, where $P_\ell$ is the $\ell$th prime{\label{alg:main1:detq}}
		\State List out $\mathcal{P}_{|I|}:=\left\{P_\ell \colon \ell=q,\ldots, q+\max\left(0,2(|I|-1)\ceil{-1 + \log_{P_q}(\tilde{M})}-1\right)\right\}$ with\label{alg:main1:detPs}
			\item[]\qquad$\tilde{M}:=\max_{k\in J'_0}{k}-\min_{h\in J'_0}{h}+1$
		\State Determine primes $\tilde{P}_0,\ldots,\tilde{P}_{L-1}\in\mathcal{P}_{|I|}$ such that
			$|J'_{r+1}|\le |J'_{r}|/2$ with
			\item[]\qquad 
			$
			J'_{r+1}:=
			\left\{h\in J'_r\,\colon h\equiv h'\imod{\tilde{P}_r}\text{ for at least one } h'\in J'_0\setminus\{h\}\right\}
			$ \; // cf.~\cref{eq:def_Jprime_r} 
			\label{alg:main1:detIprime}
		\EndProcedure
		\hrule		
		\item[]\rule{0pt}{1.5em}{\bfseries Runtime Complexity:}{ $\OO{|I|^2\,\log|I|\, \log{\tilde{M}}+\;|I|\left(d+(\log{\tilde{M}})\log(\log{\tilde{M}})\right)}$}
	\end{algorithmic}
\end{algorithm}

The approach analyzed in \cref{thm:not_more_than_half_of_frequencies_collide} provides a completely constructive, deterministic method for building reconstructing multiple rank\mbox{-}1 lattices from reconstructing single rank\mbox{-}1 lattices.
\Cref{alg:main1} summarizes the suggested approach in detail.
In the following, we analyze the runtime complexity.

We start by analyzing Line~\ref{alg:main1:innerprod} which is obviously in $\OO{d\,|I|}$.
The arithmetic complexity of Lines~\ref{alg:main1:detq} and \ref{alg:main1:detPs} are dominated by determining the
set of primes $\mathcal{P}_{|I|}$, which can be done in linear time with respect to $P_{q+K-1}\le C_1|I|(\log_{|I|}\tilde{M})\ln (C_2|I|\log_{|I|}\tilde{M})$ estimated in \cref{lem:estimate_PqKm1}, therefore requiring $\OO{|I|(\log{\tilde{M}})\log(\log{\tilde{M}})}$ arithmetic operations.

Line~\ref{alg:main1:detIprime} can be realized using two loops. In worst case, we have to determine the sets
\begin{equation}
\left\{h\in J_r'\,\colon h\equiv h'\imod{P_{\ell}}\text{ for at least one } h'\in J_0'\setminus\{h\}\right\}
\label{eq:alg:determineJell}
\end{equation}
for $r \in \{0,\ldots,\floor{\log_2s}\}$ and $\ell\in\{q,\ldots,q+\max\left(0,2(|I|-1)\ceil{-1 + \log_{P_q}(\tilde{M})}-1\right)\}$.
For fixed $r$ and $\ell$, we can determine~\cref{eq:alg:determineJell} in $\OO{|I|\log |I|}$, which yields
a total arithmetic complexity of $\OO{|I|^2\log|I|\log{\tilde{M}}}$ for Line~\ref{alg:main1:detIprime}.
Altogether, we observe a runtime complexity as stated in \cref{alg:main1}.

In the following, we comment on practical issues of \cref{alg:main1}.
Line~\ref{alg:main1:innerprod} might suffer from overflowing integers which can be
avoided by using higher precision integer representations. An alternative is to skip this precomputation and instead compute the inner products modulo $ P_{ \ell } $ on the fly in Line~\ref{alg:main1:detIprime}
which will increase the runtime complexity by a factor of~$d$ in the first summand. Note also that one does not necessarily need to compute~$\tilde{M}$ in advance. For the checks in Line~\ref{alg:main1:detIprime}, one might just start with the prime $P_q$ and increase the prime number using some \texttt{nextprime} function, which would increase the second summand in the runtime complexity.

Finally, we discuss the range of the numbers $\tilde{M}$ as well as the influence of the 
original single rank\mbox{-}1 lattice on the estimates in this paper.
In general, there are two different suitable approaches for finding a reconstructing single rank\mbox{-}1 lattice for
a given frequency index set~$I$. 
A simple approach is to just pick a rank\mbox{-}1 lattice $\Lambda(\boldz,M)$ that provides the reconstruction property from a simple number-theoretic point of view. For instance one can choose generating vectors~$\boldz$ and lattice sizes~$M$ that fulfill
\begin{align*}
z_1\in\N,\quad z_{i} &\ge (1+\max_{\boldk\in I}k_{i-1} - \min_{\boldh\in I}h_{i-1})z_{i-1}, \quad i=2,\ldots,d,\\
M&\ge (1+\max_{\boldk\in I}k_{d} - \min_{\boldh\in I}h_d)z_{d}.
\end{align*}
Clearly, even for extremely sparse frequency sets and moderate expansions of $I$ this approach will
lead to exponentially increasing $d$th components $z_d\ge 2^{d-1}$ and lattice sizes $M\ge 2^d$ even for
$\min_{j=1}^d(\max_{\boldk\in I}k_{j} - \min_{\boldh\in I}h_j)\ge 1$.

As in \cref{rem:discuss_tildeM}, this approach will lead to exponential increase in $\tilde{M}$ and
thus a linear dependence of the dimension $ d $ from $\log{\tilde{M}}$. From a theoretical point of view, this turns out to be disadvantageous for higher dimensions~$d$ due to the fact that the runtime complexity of \cref{alg:main1} as well as 
the estimates of the total number of sampling values in \cref{thm:not_more_than_half_of_frequencies_collide,theorem:lattice_reduction} will be affected by this factor.

A more costly way of determining reconstructing single rank\mbox{-}1 lattices is a suitable CBC construction as suggested in~\cite{kuo2019function}, which requires a computational complexity in $\OO{ds^2}$. The additional computational effort pays off when applying the theoretical bounds on the resulting lattice size $M$ to the estimates of this paper. In more detail, the CBC approach offers reconstructing rank\mbox{-}1 lattices with prime lattice sizes $M$ bounded from above by $M\le \max(s^2,2(N_I+1))$, cf.\ \cite{Kae2013, kuo2019function}.
As a consequence, the estimates in \cref{rem:discuss_tildeM} give $\tilde{M}\le CdN_I^2s^2$ or even $\tilde{M}\le C'R N_I s^2$ for $I$ a subset of an $\ell_1$-ball of radius $R$.
Thus, the estimates  on the required number of sampling values for unique reconstruction of multivariate trigonometric polynomials in $\Pi_I$, cf.~\cref{eq:details_sampling_nodes}, are respectively only either logarithmically dependent on $ d $ or even independent of $ d $.

\section{Numerical tests}

In this section, we investigate the statements of \cref{thm:not_more_than_half_of_frequencies_collide,theorem:lattice_reduction} numerically\footnote{All code is available at \url{https://www.math.msu.edu/~markiwen/Code.html}}. We consider different types of frequency sets~$I$. In particular, we use symmetric hyperbolic cross type frequency sets
\begin{equation}
I=
H_{R,\text{even}}^d := \left\{\boldk:=(k_1,\ldots,k_d)^\top\in(2\Z)^d \colon \prod_{t\in\{1,\ldots,d\}} \max(1,|k_t|) \leq R \right\}
\label{equ:def:H_R_even}
\end{equation}
with expansion parameter $R\in\N$, which results in $N_I \leq 2R$, in up to $d=9$ spatial dimensions. These frequency sets $H_{R,\text{even}}^d$ have the property that in each frequency component only even indices occur.
This matches the behavior of the Fourier support of the test function~$G_3^d$ introduced below in \cref{sec:numerics:symhceven:approx} which we approximate using samples on multiple rank-1 lattices, see also \cite{KaPoVo13, KaVo19} and \cite[section~2.3.5]{volkmerdiss}.

In addition, we use random frequency sets $I\subset ([-R,R]\cap\Z)^d$, which yield $N_I \leq 2R$, and we consider these in up to $d=10\,000$ spatial dimensions.

\subsection{Deterministic multiple rank-1 lattices generated by Algorithm~\ref{alg:main1} suitable for reconstruction and approximation}
\label{sec:numerics:lattices_alg1}

\subsubsection{Resulting numbers of samples and oversampling factors}

In the beginning, we determine the overall number of samples when applying \cref{alg:main1}. Up to an additive term of $1-L$, this corresponds to $\sum_{\ell=0}^{L-1} \tilde{P}_\ell$ in \cref{thm:not_more_than_half_of_frequencies_collide}, since the node~$\boldzero$ (point of origin) is contained in each of the resulting rank-1 lattices $\Lambda(\boldz,\tilde{P}_\ell)$.
We start with symmetric hyperbolic cross sets $I=H_{R,\text{even}}^d$ as defined in~\eqref{equ:def:H_R_even} and consider three different types of reconstructing single rank-1 lattices $\Lambda(\boldz,M,I)$ as input for \cref{alg:main1}.

First, we use the rank-1 lattices from \cite[Table~6.1]{KaPoVo13}, which were generated by the CBC method \cite[Algorithm~3.7]{kaemmererdiss}, as input for \cref{alg:main1}. We plot the results in \cref{fig:symhceven:samples:cbc} for spatial dimensions $d\in\{2,3,\ldots,9\}$ and with various refinements $R\in\N$ of $I=H_{R,\text{even}}^d$. 
The observed numbers of samples seem to behave slightly worse than linear with respect to the cardinality of the frequency set~$I$. The corresponding theoretical upper bounds according to \cref{thm:not_more_than_half_of_frequencies_collide} using~\eqref{eq:estimate2_tildeM} for $\tilde{M}$ are also shown as filled markers with dashed lines for spatial dimensions $d\in\{2,9\}$ in \cref{fig:symhceven:samples:cbc}. The plotted upper bounds are distinctly larger and their slopes seem to be slightly higher than those observed by plotting the numerical tests.

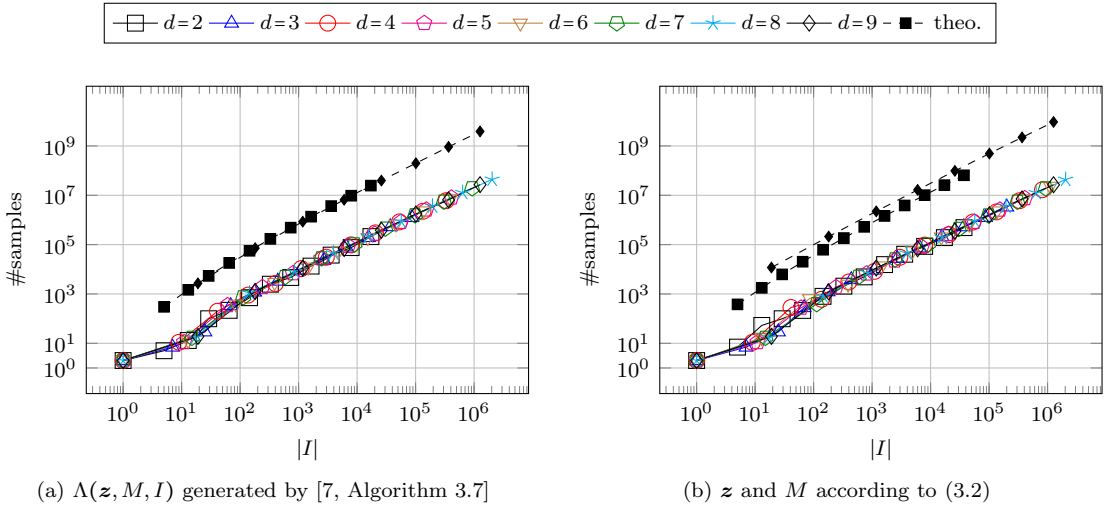
\begin{figure}[!ht]
	\begin{center}
	\begin{tikzpicture}
	\begin{axis}[font=\footnotesize,hide axis,xmin=0,xmax=1,ymin=0,ymax=1,width=\textwidth,legend style={legend cell align=left},legend columns=9]
	\addlegendimage{black,mark=square,mark size=3pt,mark options={solid}}
	\addlegendentry{$d\!=\!2$}
	\addlegendimage{blue,mark=triangle,mark size=3pt,mark options={solid}}
	\addlegendentry{$d\!=\!3$}
	\addlegendimage{red,mark=o,mark size=3pt,mark options={solid}}
	\addlegendentry{$d\!=\!4$}
	\addlegendimage{magenta,mark=pentagon,mark size=3pt,mark options={solid}}
	\addlegendentry{$d\!=\!5$}
	\addlegendimage{brown,mark=triangle,mark size=3pt,mark options={solid,rotate=180}}
	\addlegendentry{$d\!=\!6$}
	\addlegendimage{darkgreen,mark=pentagon,mark size=3pt,mark options={solid,rotate=180}}
	\addlegendentry{$d\!=\!7$}
	\addlegendimage{cyan,mark=star,mark size=3pt,mark options={solid}}
	\addlegendentry{$d\!=\!8$}
	\addlegendimage{black,mark=diamond,mark size=3pt,mark options={solid}}
	\addlegendentry{$d\!=\!9$}
	\addlegendimage{black,mark=square*,mark size=2pt,mark options={solid},dashed}
	\addlegendentry{theo.}
	\end{axis}
	\end{tikzpicture}
	\end{center}
	\vspace{-0.5em}
	\noindent
	\subfloat[{$\Lambda(\boldz,M,I)$ generated by \cite[Algorithm~3.7]{kaemmererdiss}}]{\label{fig:symhceven:samples:cbc}
		\begin{tikzpicture}[baseline]
		\begin{axis}[
		font=\footnotesize,
		enlarge x limits=true,
		enlarge y limits=true,
		height=0.38\textwidth,
		grid=major,
width=0.5\textwidth,
		xmin=1,xmax=2e6,
		ymin=1,ymax=2.41e10,
ytick={1,10,1000,100000,1e7,1e9},
xmode=log,
		ymode=log,
		xlabel={$|I|$},
		ylabel={\#samples},
		legend style={legend cell align=left,at={(1,1.43)}},
legend columns = 3,
		]
		\addplot[black,mark=square,mark size=3pt,mark options={solid}] coordinates {
(1,2) (5,5) (13,13) (29,99) (65,217) (145,707) (329,2487) (733,4655) (1633,13481) (3605,37563) (7913,75829) (17217,212081)
		};
\addplot[blue,mark=triangle,mark size=3pt,mark options={solid}] coordinates {
(1,2) (7,7) (25,29) (69,335) (177,1193) (441,3305) (1097,9183) (2693,23019) (6529,80807) (15645,192939) (37025,532699) (86593,1248219)
		};
\addplot[red,mark=o,mark size=3pt,mark options={solid}] coordinates {
(1,2) (9,11) (41,199) (137,901) (401,2643) (1105,10667) (2977,30985) (7897,97885) (20609,299789) (52953,833627) (133905,2252871) (333457,6049405)
		};
\addplot[magenta,mark=pentagon,mark size=3pt,mark options={solid}] coordinates {
(1,2) (11,11) (61,347) (241,1879) (801,6815) (2433,27465) (7073,87881) (20073,273487) (55873,827399) (152713,2619383) (409825,7984285)
		};
\addplot[brown,mark=triangle,mark size=3pt,mark options={solid,rotate=180}] coordinates {
(1,2) (13,13) (85,387) (389,2931) (1457,12273) (4865,55887) (15241,192345) (46069,682999) (135905,2336985) (392717,7276429)
		};
\addplot[darkgreen,mark=pentagon,mark size=3pt,mark options={solid,rotate=180}] coordinates {
(1,2) (15,17) (113,745) (589,4977) (2465,28313) (9017,103899) (30409,451581) (97709,1569119) (304321,5632663) (925445,19558255)
		};
\addplot[cyan,mark=star,mark size=3pt,mark options={solid}] coordinates {
(1,2) (17,17) (145,1081) (849,8093) (3937,45429) (15713,216701) (56961,853129) (194353,3609763) (637697,12776127) (2034289,43889827)
		};
\addplot[black,mark=diamond,mark size=3pt,mark options={solid}] coordinates {
(1,2) (19,19) (181,1267) (1177,11073) (6001,69775) (26017,381579) (101185,1741509) (366289,6532405) (1264513,27025383)
		};
\addplot[black,mark=square*,mark size=2pt,mark options={solid},dashed] coordinates {
(5,3.004448e+02) (13,1.472826e+03) (29,5.373409e+03) (65,1.800891e+04) (145,5.622379e+04) (329,1.704671e+05) (733,4.881846e+05) (1633,1.358844e+06) (3605,3.660925e+06) (7913,9.623816e+06) (17217,2.468241e+07)
};
\addplot[forget plot,black,mark=diamond*,mark size=2pt,mark options={solid}, dashed] coordinates {
(19,2.699610e+03) (181,7.181055e+04) (1177,8.541646e+05) (6001,6.571028e+06) (26017,3.945211e+07) (101185,1.993941e+08) (366289,9.171609e+08) (1264513,3.880749e+09)
		};		
		\end{axis}
		\end{tikzpicture}
	}
\hfill
	\subfloat[$\boldz$ and $M$ according to~\eqref{eq:lat1}]{\label{fig:symhceven:samples:lat1}
		\begin{tikzpicture}[baseline]
		\begin{axis}[
		font=\footnotesize,
		enlarge x limits=true,
		enlarge y limits=true,
		height=0.38\textwidth,
		grid=major,
width=0.5\textwidth,
		xmin=1,xmax=2e6,
		ymin=1,ymax=2.41e10,
ytick={1,10,1000,100000,1e7,1e9},
xmode=log,
		ymode=log,
		xlabel={$|I|$},
		ylabel={\#samples},
		legend style={legend cell align=left,at={(1,1.43)}},
legend columns = 3,
		]
		\addplot[black,mark=square,mark size=3pt,mark options={solid}] coordinates {
(1,2) (5,7) (13,53) (29,99) (65,215) (145,801) (329,2071) (733,4929) (1633,15495) (3605,40725) (7913,84909) (17217,192081) (37241,489031)
		};
\addplot[blue,mark=triangle,mark size=3pt,mark options={solid}] coordinates {
(1,2) (7,7) (25,29) (69,279) (177,1225) (441,3811) (1097,9255) (2693,27637) (6529,73385) (15645,207443) (37025,529149) (86593,1256205) (200225,3349227)
		};
\addplot[red,mark=o,mark size=3pt,mark options={solid}] coordinates {
(1,2) (9,11) (41,283) (137,607) (401,3075) (1105,8275) (2977,30681) (7897,96145) (20609,255887) (52953,824761) (133905,2129745) (333457,6037843) (818449,16829701)
		};
\addplot[magenta,mark=pentagon,mark size=3pt,mark options={solid}] coordinates {
(1,2) (11,11) (61,275) (241,1877) (801,6671) (2433,22383) (7073,80365) (20073,269495) (55873,823821) (152713,2600909) (409825,7529071)
		};
\addplot[brown,mark=triangle,mark size=3pt,mark options={solid,rotate=180}] coordinates {
(1,2) (13,17) (85,645) (389,2491) (1457,13731) (4865,55463) (15241,204513) (46069,676435) (135905,2316725) (392717,7222379) (1112313,23206405)
		};
\addplot[darkgreen,mark=pentagon,mark size=3pt,mark options={solid,rotate=180}] coordinates {
(1,2) (15,17) (113,393) (589,5043) (2465,25515) (9017,102115) (30409,441557) (97709,1544269) (304321,5600069) (925445,19312573) };
\addplot[cyan,mark=star,mark size=3pt,mark options={solid}] coordinates {
(1,2) (17,17) (145,811) (849,7089) (3937,44771) (15713,197643) (56961,898975) (194353,3552845) (637697,12583473) (2034289,43452535)
		};
\addplot[black,mark=diamond,mark size=3pt,mark options={solid}] coordinates {
(1,2) (19,19) (181,1213) (1177,11263) (6001,67969) (26017,352333) (101185,1615519) (366289,7133211) (1264513,26693535)
		};
\addplot[black,mark=square*,mark size=2pt,mark options={solid},dashed] coordinates {
(5,3.782332e+02) (13,1.767574e+03) (29,6.215233e+03) (65,2.027008e+04) (145,6.203287e+04) (329,1.853090e+05) (733,5.247574e+05) (1633,1.447951e+06) (3605,3.874084e+06) (7913,1.012727e+07) (17217,2.585406e+07) (37241,6.480749e+07)
};
\addplot[forget plot,black,mark=diamond*,mark size=2pt,mark options={solid}, dashed] coordinates {
(19,1.177435e+04) (181,2.125965e+05) (1177,2.232009e+06) (6001,1.656532e+07) (26017,9.759745e+07) (101185,4.917523e+08) (366289,2.227825e+09) (1264513,9.378038e+09)
		};		
		\end{axis}
		\end{tikzpicture}
	}
	\caption{Overall \#samples $=1-L+\sum_{\ell=0}^{L-1} \tilde{P}_\ell$ for symmetric hyperbolic cross index sets $I=H_{R,\text{even}}^d$. Filled markers with dashed lines represent theoretical upper bounds from \cref{thm:not_more_than_half_of_frequencies_collide} for $d\in\{2,9\}$ calculated using \cref{eq:estimate2_tildeM}.}\label{fig:symhceven:samples}
\end{figure}

Second, we consider reconstructing single rank-1 lattices $\Lambda(\boldz,M,I)$ with
\begin{equation}\label{eq:lat1}
\boldz:=(1,N_I+1,(N_I+1)^2,\ldots,(N_I+1)^{d-1})^\top \text{ and } M:=(N_I+1)^d=(2R+1)^d,
\end{equation}
where $N_I=2R$ in our case, and we show the results in \cref{fig:symhceven:samples:lat1}. We observe that the obtained numbers of samples are similar to the ones in \cref{fig:symhceven:samples:cbc}, and the theoretical upper bounds according to \cref{thm:not_more_than_half_of_frequencies_collide} using~\eqref{eq:estimate2_tildeM} for $\tilde{M}$ are slightly higher due to increased components of the generating vector~$\boldz$ of the reconstructing single rank-1 lattices.

Third, we apply \cref{alg:main1} to reconstructing single rank-1 lattices $\Lambda(\boldz,M,I)$ as considered in \cite[section~6]{Iw13}. In detail, we choose
\begin{equation}\label{eq:lat2}
\begin{split}
& M:=\prod_{t\in\{1,2,\ldots,d\}} q_t \text{ and } \boldz:=(M/q_1,M/q_2,\ldots,M/q_d)^\top, \\ &\text{where } q_1 := d N_I+d+1 \text{ and } q_{t+1}:=\min\{p\in\N\colon p > q_t \text{ and } p \text{ prime}\}.
\end{split}
\end{equation}
Here, the observed numerical results yield plots that do not differ recognizably from \cref{fig:symhceven:samples:lat1}, and we therefore omit these plots. We would like to point out, that the theoretical upper bounds for that kind of reconstructing single rank-1 lattices are slightly worse than those plotted
in \cref{fig:symhceven:samples:lat1}, cf.\,\cref{rem:discuss_tildeM}.

Note that when running \cref{alg:main1} for single rank-1 lattices $\Lambda(\boldz,M,I)$ of type~\eqref{eq:lat1} and~\eqref{eq:lat2} in practice, one may need to deal with limited numeric precision in the computer arithmetic. For instance, for higher spatial dimensions, some components $z_t$ of the generating vector $\boldz$ may become larger than 64-bit integers. This means that the sets $J'_r$ may have to be computed carefully and repeatedly modulo each considered prime $P\in\mathcal{P}_{|I|}$ when searching for the primes $\tilde{P}_0,\ldots,\tilde{P}_{L-1}$ in Line~\ref{alg:main1:detIprime} of \cref{alg:main1}.

In order to have a closer look at the number of samples, we visualize the oversampling factor \#samples$\,/\,|I| = (1-L+\sum_{\ell=0}^{L-1} \tilde{P}_\ell)/|I|$ in \cref{fig:samples:symhceven:oversampling}. For the considered test cases and the three different types of lattices, we observe that the oversampling factors are below $1.7\,\ln |I| + 3$ for $|I|>1$. This is distinctly smaller than the theoretical upper bounds in \cref{thm:not_more_than_half_of_frequencies_collide} suggest, which have a constant of $\approx 5.7$ and additional logarithmic factors depending on~$\tilde{M}$. For instance in \cref{fig:samples:symhceven:oversampling:cbc}, for $I=H_{256,\text{even}}^9$ (cardinality $|I|=1\,264\,513$ and \#samples $=27\,025\,383$), the oversampling factor is $\approx 21.37$ whereas the corresponding upper bound for the oversampling factor is $\approx 3\,069$ according to \cref{thm:not_more_than_half_of_frequencies_collide} using~\eqref{eq:estimate2_tildeM} for $\tilde{M}$. The plots for reconstructing single rank-1 lattices $\Lambda(\boldz,M,I)$ according to~\eqref{eq:lat2} look similar to the ones according to~\eqref{eq:lat1}, where the latter are shown in \cref{fig:samples:symhceven:oversampling:lat1}. Moreover, we only observe a relatively small difference compared to \cref{fig:samples:symhceven:oversampling:cbc}.

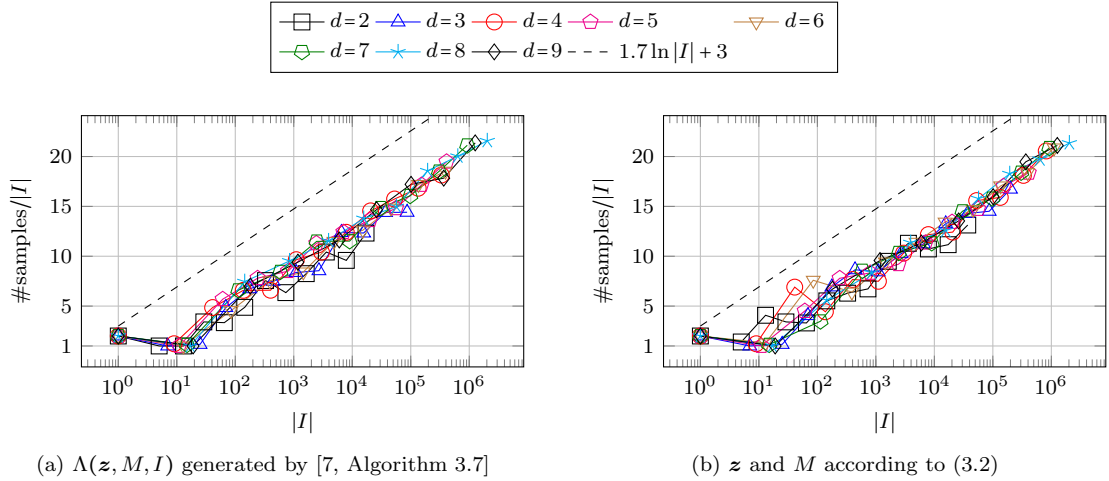
\begin{figure}[!ht]
	\begin{center}
	\begin{tikzpicture}
	\begin{axis}[font=\footnotesize,hide axis,xmin=0,xmax=1,ymin=0,ymax=1,width=\textwidth,legend style={legend cell align=left},legend columns=5]
	\addlegendimage{black,mark=square,mark size=3pt,mark options={solid}}
	\addlegendentry{$d\!=\!2$}
	\addlegendimage{blue,mark=triangle,mark size=3pt,mark options={solid}}
	\addlegendentry{$d\!=\!3$}
	\addlegendimage{red,mark=o,mark size=3pt,mark options={solid}}
	\addlegendentry{$d\!=\!4$}
	\addlegendimage{magenta,mark=pentagon,mark size=3pt,mark options={solid}}
	\addlegendentry{$d\!=\!5$}
	\addlegendimage{brown,mark=triangle,mark size=3pt,mark options={solid,rotate=180}}
	\addlegendentry{$d\!=\!6$}
	\addlegendimage{darkgreen,mark=pentagon,mark size=3pt,mark options={solid,rotate=180}}
	\addlegendentry{$d\!=\!7$}
	\addlegendimage{cyan,mark=star,mark size=3pt,mark options={solid}}
	\addlegendentry{$d\!=\!8$}
	\addlegendimage{black,mark=diamond,mark size=3pt,mark options={solid}}
	\addlegendentry{$d\!=\!9$}
	\addlegendimage{black,domain=1:1e6, samples=100, dashed}
	\addlegendentry{$1.7\, \ln |I|+3$}
	\end{axis}
	\end{tikzpicture}
	\end{center}
	\vspace{-0.5em}
\noindent
\subfloat[{$\Lambda(\boldz,M,I)$ generated by \cite[Algorithm~3.7]{kaemmererdiss}}]{\label{fig:samples:symhceven:oversampling:cbc}
		\begin{tikzpicture}[baseline]
\begin{axis}[
font=\footnotesize,
enlarge x limits=true,
enlarge y limits=true,
height=0.33\textwidth,
grid=major,
width=0.5\textwidth,
xmin=1,xmax=2e6,
ymin=1,ymax=22,
ytick={1,5,10,15,20},
xmode=log,
xlabel={$|I|$},
ylabel={\#samples$/|I|$},
legend style={legend cell align=left,at={(1.0,1.53)}},
legend columns = 3,
]
\addplot[black,mark=square,mark size=3pt,mark options={solid}] coordinates {
(1,2.000e+00) (5,1.000e+00) (13,1.000e+00) (29,3.414e+00) (65,3.338e+00) (145,4.876e+00) (329,7.559e+00) (733,6.351e+00) (1633,8.255e+00) (3605,1.042e+01) (7913,9.583e+00) (17217,1.232e+01)
};
\addplot[blue,mark=triangle,mark size=3pt,mark options={solid}] coordinates {
(1,2.000e+00) (7,1.000e+00) (25,1.160e+00) (69,4.855e+00) (177,6.740e+00) (441,7.494e+00) (1097,8.371e+00) (2693,8.548e+00) (6529,1.238e+01) (15645,1.233e+01) (37025,1.439e+01) (86593,1.441e+01)
};
\addplot[red,mark=o,mark size=3pt,mark options={solid}] coordinates {
(1,2.000e+00) (9,1.222e+00) (41,4.854e+00) (137,6.577e+00) (401,6.591e+00) (1105,9.653e+00) (2977,1.041e+01) (7897,1.240e+01) (20609,1.455e+01) (52953,1.574e+01) (133905,1.682e+01) (333457,1.814e+01)
};
\addplot[magenta,mark=pentagon,mark size=3pt,mark options={solid}] coordinates {
(1,2.000e+00) (11,1.000e+00) (61,5.689e+00) (241,7.797e+00) (801,8.508e+00) (2433,1.129e+01) (7073,1.242e+01) (20073,1.362e+01) (55873,1.481e+01) (152713,1.715e+01) (409825,1.948e+01)
};
\addplot[brown,mark=triangle,mark size=3pt,mark options={solid,rotate=180}] coordinates {
(1,2.000e+00) (13,1.000e+00) (85,4.553e+00) (389,7.535e+00) (1457,8.423e+00) (4865,1.149e+01) (15241,1.262e+01) (46069,1.483e+01) (135905,1.720e+01) (392717,1.853e+01)
};
\addplot[darkgreen,mark=pentagon,mark size=3pt,mark options={solid,rotate=180}] coordinates {
(1,2.000e+00) (15,1.133e+00) (113,6.593e+00) (589,8.450e+00) (2465,1.149e+01) (9017,1.152e+01) (30409,1.485e+01) (97709,1.606e+01) (304321,1.851e+01) (925445,2.113e+01)
};
\addplot[cyan,mark=star,mark size=3pt,mark options={solid}] coordinates {
(1,2.000e+00) (17,1.000e+00) (145,7.455e+00) (849,9.532e+00) (3937,1.154e+01) (15713,1.379e+01) (56961,1.498e+01) (194353,1.857e+01) (637697,2.003e+01) (2034289,2.158e+01)
};
\addplot[black,mark=diamond,mark size=3pt,mark options={solid}] coordinates {
(1,2.000e+00) (19,1.000e+00) (181,7.000e+00) (1177,9.408e+00) (6001,1.163e+01) (26017,1.467e+01) (101185,1.721e+01) (366289,1.783e+01) (1264513,2.137e+01)
};
\addplot[black,domain=1:1e6, samples=100, dashed]{1.7*ln(x)+3};
\end{axis}
\end{tikzpicture}
}
\hfill
\subfloat[$\boldz$ and $M$ according to~\eqref{eq:lat1}]{\label{fig:samples:symhceven:oversampling:lat1}
\begin{tikzpicture}[baseline]
	\begin{axis}[
	font=\footnotesize,
	enlarge x limits=true,
	enlarge y limits=true,
	height=0.33\textwidth,
grid=major,
width=0.5\textwidth,
xmin=1,xmax=2e6,
ymin=1,ymax=22,
ytick={1,5,10,15,20},
xmode=log,
	xlabel={$|I|$},
	ylabel={\#samples$/|I|$},
	legend style={legend cell align=left,at={(1.0,1.53)}},
	legend columns = 3,
]
	\addplot[black,mark=square,mark size=3pt,mark options={solid}] coordinates {
(1,2.000e+00) (5,1.400e+00) (13,4.077e+00) (29,3.414e+00) (65,3.308e+00) (145,5.524e+00) (329,6.295e+00) (733,6.724e+00) (1633,9.489e+00) (3605,1.130e+01) (7913,1.073e+01) (17217,1.116e+01) (37241,1.313e+01)
	};
\addplot[blue,mark=triangle,mark size=3pt,mark options={solid}] coordinates {
(1,2.000e+00) (7,1.000e+00) (25,1.160e+00) (69,4.043e+00) (177,6.921e+00) (441,8.642e+00) (1097,8.437e+00) (2693,1.026e+01) (6529,1.124e+01) (15645,1.326e+01) (37025,1.429e+01) (86593,1.451e+01) (200225,1.673e+01)
	};
\addplot[red,mark=o,mark size=3pt,mark options={solid}] coordinates {
(1,2.000e+00) (9,1.222e+00) (41,6.902e+00) (137,4.431e+00) (401,7.668e+00) (1105,7.489e+00) (2977,1.031e+01) (7897,1.217e+01) (20609,1.242e+01) (52953,1.558e+01) (133905,1.590e+01) (333457,1.811e+01) (818449,2.056e+01)
	};
\addplot[magenta,mark=pentagon,mark size=3pt,mark options={solid}] coordinates {
(1,2.000e+00) (11,1.000e+00) (61,4.508e+00) (241,7.788e+00) (801,8.328e+00) (2433,9.200e+00) (7073,1.136e+01) (20073,1.343e+01) (55873,1.474e+01) (152713,1.703e+01) (409825,1.837e+01)
	};
\addplot[brown,mark=triangle,mark size=3pt,mark options={solid,rotate=180}] coordinates {
(1,2.000e+00) (13,1.308e+00) (85,7.588e+00) (389,6.404e+00) (1457,9.424e+00) (4865,1.140e+01) (15241,1.342e+01) (46069,1.468e+01) (135905,1.705e+01) (392717,1.839e+01) (1112313,2.086e+01)
	};
\addplot[darkgreen,mark=pentagon,mark size=3pt,mark options={solid,rotate=180}] coordinates {
(1,2.000e+00) (15,1.133e+00) (113,3.478e+00) (589,8.562e+00) (2465,1.035e+01) (9017,1.132e+01) (30409,1.452e+01) (97709,1.580e+01) (304321,1.840e+01) (925445,2.087e+01) };
\addplot[cyan,mark=star,mark size=3pt,mark options={solid}] coordinates {
(1,2.000e+00) (17,1.000e+00) (145,5.593e+00) (849,8.350e+00) (3937,1.137e+01) (15713,1.258e+01) (56961,1.578e+01) (194353,1.828e+01) (637697,1.973e+01) (2034289,2.136e+01)
	};
\addplot[black,mark=diamond,mark size=3pt,mark options={solid}] coordinates {
(1,2.000e+00) (19,1.000e+00) (181,6.702e+00) (1177,9.569e+00) (6001,1.133e+01) (26017,1.354e+01) (101185,1.597e+01) (366289,1.947e+01) (1264513,2.111e+01)
	};
\addplot[black,domain=1:1e6, samples=100, dashed]{1.7*ln(x) + 3};
\end{axis}
	\end{tikzpicture}
}
\caption{Oversampling factors for deterministic reconstructing multiple rank-1 lattices for symmetric hyperbolic cross index sets~$H_{R,\text{even}}^d$.}\label{fig:samples:symhceven:oversampling}
\end{figure}

Next, we change the setting and use frequency sets $I$ drawn uniformly randomly from cubes $[-R,R]^d\cap\Z^d$. We generate reconstructing single rank-1 lattices $\Lambda(\boldz,M,I)$ using \cite[Algorithm~3.7]{kaemmererdiss}.
Then, we apply \cref{alg:main1} in order to deterministically generate reconstructing multiple rank-1 lattices. We repeat the test 10 times for each setting with newly randomly chosen frequency sets~$I$ and determine the maximum number of samples over the 10 repetitions. For frequency set sizes $|I|\in\{10,100,1\,000,10\,000\}$ in $d\in\{2,3,4,6,10,100,1\,000,10\,000\}$ spatial dimensions and additionally $|I|=100\,000$ for some of the aforementioned spatial dimensions $d$, we visualize the resulting oversampling factors in \cref{fig:samples:rand:oversampling} for expansion parameter $R=64$ ($N_I\leq 128$). Using different reconstructing single rank-1 lattices $\Lambda(\boldz,M, I)$ as in \cref{fig:samples:symhceven:oversampling} changes the oversampling factors only slightly, and the oversampling factors are still well below $1.7\, \ln |I| + 3$, compare \cref{fig:samples:rand:oversampling:cbc,fig:samples:rand:oversampling:lat1}. The plots for reconstructing single rank\mbox{-}1 lattices $\Lambda(\boldz,M,I)$ according to~\eqref{eq:lat2} are omitted since they look very similar to \cref{fig:samples:rand:oversampling:lat1}.
As mentioned before, we have to take care of possible issues with numeric precision when running \cref{alg:main1} on reconstructing single rank-1 lattices of type~\eqref{eq:lat1} and~\eqref{eq:lat2} in practice.

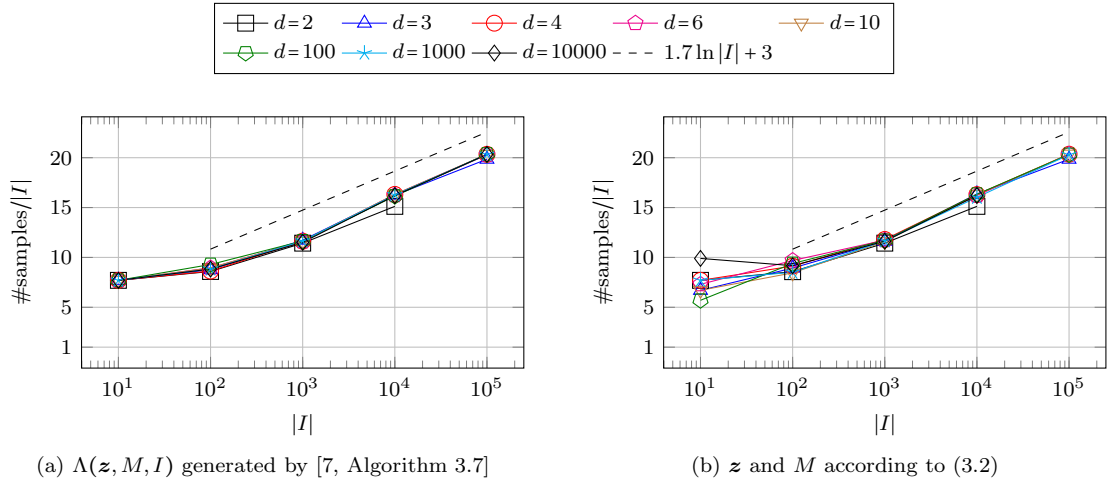
\begin{figure}[!ht]
	\begin{center}
	\begin{tikzpicture}
	\begin{axis}[font=\footnotesize,hide axis,xmin=0,xmax=1,ymin=0,ymax=1,width=\textwidth,legend style={legend cell align=left},legend columns=5]
	\addlegendimage{black,mark=square,mark size=3pt,mark options={solid}}
	\addlegendentry{$d\!=\!2$}
	\addlegendimage{blue,mark=triangle,mark size=3pt,mark options={solid}}
	\addlegendentry{$d\!=\!3$}
	\addlegendimage{red,mark=o,mark size=3pt,mark options={solid}}
	\addlegendentry{$d\!=\!4$}
	\addlegendimage{magenta,mark=pentagon,mark size=3pt,mark options={solid}}
	\addlegendentry{$d\!=\!6$}
	\addlegendimage{brown,mark=triangle,mark size=3pt,mark options={solid,rotate=180}}
	\addlegendentry{$d\!=\!10$}
	\addlegendimage{darkgreen,mark=pentagon,mark size=3pt,mark options={solid,rotate=180}}
	\addlegendentry{$d\!=\!100$}
	\addlegendimage{cyan,mark=star,mark size=3pt,mark options={solid}}
	\addlegendentry{$d\!=\!1000$}
	\addlegendimage{black,mark=diamond,mark size=3pt,mark options={solid}}
	\addlegendentry{$d\!=\!10000$}
	\addlegendimage{black,domain=100:1e5, samples=100, dashed}
	\addlegendentry{$1.7 \, \ln |I| + 3$}
	\end{axis}
	\end{tikzpicture}
\end{center}
\vspace{-0.5em}
\noindent
	\subfloat[{$\Lambda(\boldz,M,I)$ generated by \cite[Algorithm~3.7]{kaemmererdiss}}]{\label{fig:samples:rand:oversampling:cbc}
		\begin{tikzpicture}[baseline]
		\begin{axis}[
		font=\footnotesize,
		enlarge x limits=true,
		enlarge y limits=true,
		height=0.33\textwidth,
		grid=major,
		width=0.5\textwidth,
		ymin=1,ymax=22,
		ytick={1,5,10,15,20},
		xmode=log,
		xlabel={$|I|$},
		ylabel={\#samples$/|I|$},
		legend style={legend cell align=left,at={(1.0,1.53)},transpose legend},
		legend columns = 3,
		]
		\addplot[black,mark=square,mark size=3pt,mark options={solid}] coordinates {
(10,7.700e+00) (100,8.590e+00) (1000,1.143e+01) (10000,1.512e+01)
		};
\addplot[blue,mark=triangle,mark size=3pt,mark options={solid}] coordinates {
(10,7.700e+00) (100,8.770e+00) (1000,1.173e+01) (10000,1.624e+01) (100000,1.985e+01)
		};
\addplot[red,mark=o,mark size=3pt,mark options={solid}] coordinates {
(10,7.700e+00) (100,8.630e+00) (1000,1.163e+01) (10000,1.632e+01) (100000,2.034e+01)
		};
\addplot[magenta,mark=pentagon,mark size=3pt,mark options={solid}] coordinates {
(10,7.700e+00) (100,8.950e+00) (1000,1.161e+01) (10000,1.623e+01) (100000,2.029e+01)
		};
\addplot[brown,mark=triangle,mark size=3pt,mark options={solid,rotate=180}] coordinates {
(10,7.700e+00) (100,8.950e+00) (1000,1.162e+01) (10000,1.629e+01) (100000,2.032e+01)
		};
\addplot[darkgreen,mark=pentagon,mark size=3pt,mark options={solid,rotate=180}] coordinates {
(10,7.700e+00) (100,9.290e+00) (1000,1.165e+01) (10000,1.615e+01) (100000,2.036e+01)
		};
\addplot[cyan,mark=star,mark size=3pt,mark options={solid}] coordinates {
(10,7.700e+00) (100,8.830e+00) (1000,1.164e+01) (10000,1.628e+01) (100000,2.032e+01)
		};
\addplot[black,mark=diamond,mark size=3pt,mark options={solid}] coordinates {
(10,7.700e+00) (100,8.830e+00) (1000,1.151e+01) (10000,1.620e+01) (100000,2.033e+01)
		};
\addplot[black,domain=100:1e5, samples=100, dashed]{1.7*ln(x)+3};
\end{axis}
		\end{tikzpicture}
	}
	\hfill
	\subfloat[$\boldz$ and $M$ according to~\eqref{eq:lat1}]{\label{fig:samples:rand:oversampling:lat1}
		\begin{tikzpicture}[baseline]
		\begin{axis}[
		font=\footnotesize,
		enlarge x limits=true,
		enlarge y limits=true,
		height=0.33\textwidth,
		grid=major,
		width=0.5\textwidth,
		ymin=1,ymax=22,
		ytick={1,5,10,15,20},
		xmode=log,
		xlabel={$|I|$},
		ylabel={\#samples$/|I|$},
		legend style={legend cell align=left,at={(1.0,1.53)},transpose legend},
		legend columns = 3,
		]
		\addplot[black,mark=square,mark size=3pt,mark options={solid}] coordinates {
(10,7.700e+00) (100,8.570e+00) (1000,1.143e+01) (10000,1.512e+01)
		};
\addplot[blue,mark=triangle,mark size=3pt,mark options={solid}] coordinates {
(10,6.700e+00) (100,8.930e+00) (1000,1.163e+01) (10000,1.637e+01) (100000,1.985e+01)
		};
\addplot[red,mark=o,mark size=3pt,mark options={solid}] coordinates {
(10,7.700e+00) (100,9.130e+00) (1000,1.180e+01) (10000,1.630e+01) (100000,2.035e+01)
		};
\addplot[magenta,mark=pentagon,mark size=3pt,mark options={solid}] coordinates {
(10,7.300e+00) (100,9.690e+00) (1000,1.173e+01) (10000,1.621e+01) (100000,2.035e+01)
		};
\addplot[brown,mark=triangle,mark size=3pt,mark options={solid,rotate=180}] coordinates {
(10,6.700e+00) (100,8.450e+00) (1000,1.158e+01) (10000,1.630e+01) (100000,2.036e+01)
		};
\addplot[darkgreen,mark=pentagon,mark size=3pt,mark options={solid,rotate=180}] coordinates {
(10,5.700e+00) (100,9.390e+00) (1000,1.171e+01) (10000,1.633e+01) (100000,2.032e+01)
		};
\addplot[cyan,mark=star,mark size=3pt,mark options={solid}] coordinates {
(10,7.700e+00) (100,8.510e+00) (1000,1.167e+01) (10000,1.603e+01) (100000,2.029e+01)
		};
\addplot[black,mark=diamond,mark size=3pt,mark options={solid}] coordinates {
(10,9.900e+00) (100,9.170e+00) (1000,1.163e+01) (10000,1.621e+01)
		};
\addplot[black,domain=100:1e5, samples=100, dashed]{1.7*ln(x)+3};
\end{axis}
		\end{tikzpicture}
	}
	\caption{Oversampling factors for deterministic reconstructing multiple rank-1 lattices for random frequency sets $I\subset\{-64,-63,\ldots,64\}^d$.}\label{fig:samples:rand:oversampling}
\end{figure}

\FloatBarrier
\subsubsection{Improvement of numbers of samples compared to single rank-1 lattices constructed component-by-component}

For the resulting deterministic reconstructing multiple rank-1 lattices generated by \cref{alg:main1} in the previous subsection, one aspect of particular interest is the total number of nodes compared to the reconstructing single rank-1 lattices, which are given as an input to the algorithm. We investigate this in more detail for the case of lattices generated component-by-component by \cite[Algorithm~3.7]{kaemmererdiss}. These reconstructing single rank-1 lattices $\Lambda(\boldz,M,I)$ are specifically tailored to the structure of the corresponding frequency sets~$I$. We do not consider the case when \cref{alg:main1} is applied to single rank-1 lattices of type~\eqref{eq:lat1} or~\eqref{eq:lat2} as these ones are typically extremely large compared to the cardinality~$|I|$ of the frequency sets~$I$.

First, we start with symmetric hyperbolic cross index sets $I=H_{R,\text{even}}^d$ and reconstructing single rank-1 lattices $\Lambda(\boldz,M,I)$ generated by \cite[Algorithm~3.7]{kaemmererdiss}. In \cref{fig:samples:symhceven:ratio:cbc}, the obtained \#samples from \cref{fig:symhceven:samples:cbc} is divided by the size~$M$ of the single rank-1 lattice. We observe that for smaller expansion parameters~$R$ and consequently smaller cardinalities~$|I|$, the generated multiple rank\mbox{-}1 lattices still consist of more nodes than the corresponding single rank\mbox{-}1 lattices and therefore the ratio is larger than one. One main reason for this behavior is that for the component-by-component constructed single rank\mbox{-}1 lattices, the number of nodes is initially much less than the worst case upper bounds of almost $\mathcal{O}(|I|^2)$ suggest, cf.\ \cite[section~3.8.2]{kaemmererdiss} for a detailed discussion.
 Once a certain expansion~$N_I$ and cardinality~$|I|$ have been reached, the multiple rank-1 lattices outperform the single rank-1 lattices, yielding ratios around $0.1$ in \cref{fig:samples:symhceven:ratio:cbc}, i.e., \cref{alg:main1}
reduces the number of sampling nodes by 9/10.

Second, we consider the randomly generated frequency sets from \cref{fig:samples:rand:oversampling:cbc}. In \cref{fig:samples:rand:ratio:cbc}, we visualize the ratios of the number of nodes of the deterministic reconstructing multiple rank-1 lattices generated by \cref{alg:main1} over the lattice sizes $M$ of the reconstructing single rank-1 lattices generated by \cite[Algorithm~3.7]{kaemmererdiss}.
For the spatial dimensions $d\geq 4$ considered in \cref{fig:samples:rand:oversampling:cbc}, the ratios decrease rapidly for increasing cardinality~$|I|$, and we do not observe any noticeable dependence on the spatial dimension~$d$.
Note that in the case $d=2$, the ratios are close to or above one since the cube $\{-64,-63,\ldots,64\}^2$ of possible frequencies only has cardinality 16\,641 and the single rank-1 lattices already have small oversampling factors $M/|I|<16$. Similarly, in the case $d=3$ for cardinality $|I|=10^5$, the frequency set~$I$ fills approximately 1/20 of the cube $\{-64,-63,\ldots,64\}^3$ and again the low oversampling factors $M/|I|<22$ of the single rank-1 lattices are hard to beat for multiple rank-1 lattices.

\begin{figure}[!ht]
	\subfloat[symmetric hyperbolic cross index sets $I=H_{R,\text{even}}^d$]{\label{fig:samples:symhceven:ratio:cbc}
		\begin{tikzpicture}[baseline]
		\begin{axis}[
		font=\footnotesize,
		enlarge x limits=true,
		enlarge y limits=true,
		height=0.33\textwidth,
		grid=major,
width=0.48\textwidth,
		ymin=1.5e-2,ymax=1e1,
xmode=log,
		ymode=log,
		xlabel={$|I|$},
		ylabel={\#samples$/M$},
		legend style={legend cell align=left,at={(1.0,1.4)}},
		legend columns = 4,
		]
		\addplot[black,mark=square,mark size=3pt,mark options={solid}] coordinates {
(1,2.000e+00) (5,1.000e+00) (13,1.000e+00) (29,2.415e+00) (65,1.497e+00) (145,1.297e+00) (329,1.177e+00) (733,5.594e-01) (1633,4.082e-01) (3605,2.855e-01) (7913,1.444e-01) (17217,1.010e-01)
		};
		\addlegendentry{$d=2$}
		\addplot[blue,mark=triangle,mark size=3pt,mark options={solid}] coordinates {
(1,2.000e+00) (7,1.000e+00) (25,1.000e+00) (69,3.454e+00) (177,3.020e+00) (441,1.920e+00) (1097,1.779e+00) (2693,1.067e+00) (6529,9.462e-01) (15645,5.371e-01) (37025,3.850e-01) (86593,2.305e-01)
		};
		\addlegendentry{$d=3$}
		\addplot[red,mark=o,mark size=3pt,mark options={solid}] coordinates {
(1,2.000e+00) (9,1.222e+00) (41,4.061e+00) (137,3.506e+00) (401,2.179e+00) (1105,1.834e+00) (2977,1.439e+00) (7897,8.817e-01) (20609,6.170e-01) (52953,3.542e-01) (133905,2.021e-01) (333457,1.454e-01)
		};
		\addlegendentry{$d=4$}
		\addplot[magenta,mark=pentagon,mark size=3pt,mark options={solid}] coordinates {
(1,2.000e+00) (11,1.000e+00) (61,4.284e+00) (241,3.460e+00) (801,2.213e+00) (2433,1.927e+00) (7073,1.124e+00) (20073,6.769e-01) (55873,3.553e-01) (152713,2.150e-01) (409825,1.125e-01)
		};
		\addlegendentry{$d=5$}
		\addplot[brown,mark=triangle,mark size=3pt,mark options={solid,rotate=180}] coordinates {
(1,2.000e+00) (13,1.000e+00) (85,2.825e+00) (389,2.982e+00) (1457,1.777e+00) (4865,1.638e+00) (15241,8.475e-01) (46069,4.973e-01) (135905,2.869e-01) (392717,1.433e-01)
		};
		\addlegendentry{$d=6$}
		\addplot[darkgreen,mark=pentagon,mark size=3pt,mark options={solid,rotate=180}] coordinates {
(1,2.000e+00) (15,1.133e+00) (113,4.071e+00) (589,3.029e+00) (2465,2.257e+00) (9017,1.225e+00) (30409,7.863e-01) (97709,3.856e-01) (304321,2.018e-01) (925445,1.092e-01)
		};
		\addlegendentry{$d=7$}
		\addplot[cyan,mark=star,mark size=3pt,mark options={solid}] coordinates {
(1,2.000e+00) (17,1.000e+00) (145,4.239e+00) (849,2.796e+00) (3937,1.943e+00) (15713,1.172e+00) (56961,6.831e-01) (194353,3.266e-01) (637697,1.514e-01) (2034289,7.312e-02)
		};
		\addlegendentry{$d=8$}
		\addplot[black,mark=diamond,mark size=3pt,mark options={solid}] coordinates {
(1,2.000e+00) (19,1.000e+00) (181,3.851e+00) (1177,2.260e+00) (6001,1.601e+00) (26017,9.731e-01) (101185,5.611e-01) (366289,2.452e-01) (1264513,1.314e-01)
		};
		\addlegendentry{$d=9$}
		\end{axis}
		\end{tikzpicture}
	}
	\hfill
	\subfloat[$I\subset\{-64,-63,\ldots,64\}^d$ random frequency sets]{\label{fig:samples:rand:ratio:cbc}
		\begin{tikzpicture}[baseline]
		\begin{axis}[
		font=\footnotesize,
		enlarge x limits=true,
		enlarge y limits=true,
		height=0.33\textwidth,
		grid=major,
		width=0.48\textwidth,
		ymin=7e-3,ymax=1e1,
		xmode=log,
		ymode=log,
		xlabel={$|I|$},
		ylabel={\#samples$/M$},
		legend style={legend cell align=left,at={(1.0,1.52)}},
		legend columns = 3,
		]
		\addplot[black,mark=square,mark size=3pt,mark options={solid}] coordinates {
(10,2.905e+00) (100,1.010e+00) (1000,7.498e-01) (10000,9.086e+00)
		};
		\addlegendentry{$d=2$}
		\addplot[blue,mark=triangle,mark size=3pt,mark options={solid}] coordinates {
(10,3.167e+00) (100,1.054e+00) (1000,2.297e-01) (10000,1.012e-01) (100000,9.262e-01)
		};
		\addlegendentry{$d=3$}
		\addplot[red,mark=o,mark size=3pt,mark options={solid}] coordinates {
(10,3.133e+00) (100,1.174e+00) (1000,2.400e-01) (10000,4.722e-02) (100000,1.451e-02)
		};
		\addlegendentry{$d=4$}
		\addplot[magenta,mark=pentagon,mark size=3pt,mark options={solid}] coordinates {
(10,3.133e+00) (100,1.068e+00) (1000,3.048e-01) (10000,5.578e-02) (100000,7.452e-03)
		};
		\addlegendentry{$d=6$}
		\addplot[brown,mark=triangle,mark size=3pt,mark options={solid,rotate=180}] coordinates {
(10,2.478e+00) (100,1.124e+00) (1000,2.808e-01) (10000,4.341e-02) (100000,7.760e-03)
		};
		\addlegendentry{$d=10$}
		\addplot[darkgreen,mark=pentagon,mark size=3pt,mark options={solid,rotate=180}] coordinates {
(10,2.722e+00) (100,1.168e+00) (1000,2.920e-01) (10000,4.689e-02) (100000,7.742e-03)
		};
		\addlegendentry{$d=100$}
		\addplot[cyan,mark=star,mark size=3pt,mark options={solid}] coordinates {
(10,2.333e+00) (100,1.036e+00) (1000,2.155e-01) (10000,4.840e-02) (100000,8.894e-03)
		};
		\addlegendentry{$d=1000$}
		\addplot[black,mark=diamond,mark size=3pt,mark options={solid}] coordinates {
(10,2.864e+00) (100,1.166e+00) (1000,2.269e-01) (10000,4.653e-02) (100000,7.252e-03)
		};
		\addlegendentry{$d=10000$}
		\end{axis}
		\end{tikzpicture}
	}
	\caption{Ratio \#samples for deterministic reconstructing multiple rank-1 lattices suitable for approximation over lattice size $M$ of reconstructing single rank-1 lattice $\Lambda(\boldz,M,I)$, where $\Lambda(\boldz,M,I)$ was generated by \cite[Algorithm~3.7]{kaemmererdiss}.}
\end{figure}
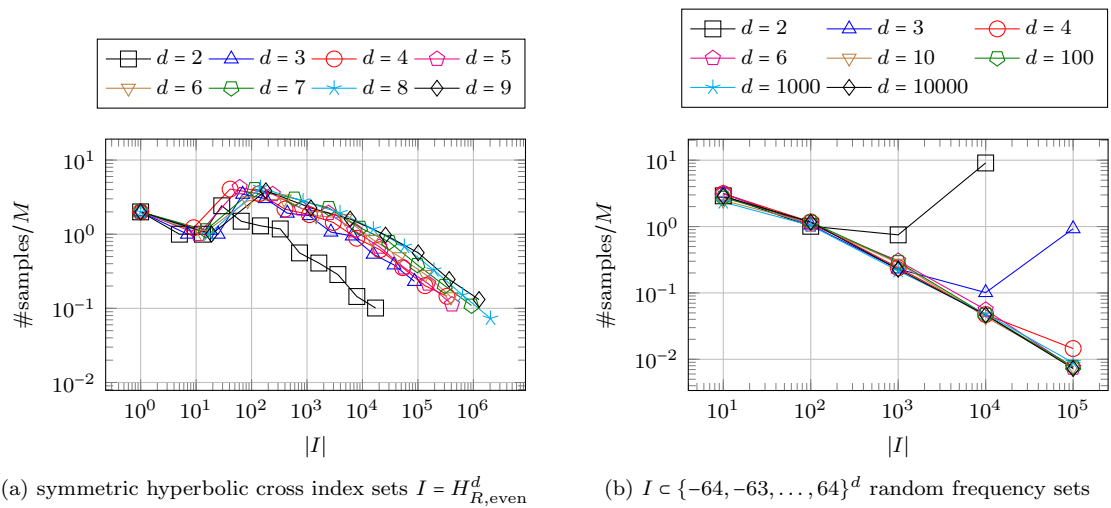

\FloatBarrier
\subsection{Comparison of reconstructing multiple and single rank-1 lattices for function approximation}\label{sec:numerics:symhceven:approx}

As mentioned in \cref{rem:ApproximationResult}, we can use the reconstruction process~\eqref{equ:MultRank1Cubature} to compute approximations of functions based on samples along rank-1 lattices.
We consider the tensor-product test functions
$G_3^d\colon\T^d\rightarrow\C$ from~\cite{KaPoVo13},
$G_3^d(\boldx):=\prod_{s=1}^{d} g_3(x_s)$,
where the one-dimensional function $g_3\colon\T\rightarrow\C$ is defined by
\begin{equation*}
g_3(x):= 4 \, \sqrt{\frac{3\pi}{207 \pi - 256}} \left(2+\sgn((x\bmod 1)-1/2) \, \sin(2\pi x)^3\right)
\end{equation*}
and
$\Vert G_3^d \Vert_{L_2(\T^d)}= 1$.
The function $G_3^d$ lies in a so-called Sobolev space of dominating mixed smoothness with smoothness almost $3.5$ such that its Fourier coefficients $(\widehat{G_3^d})_\boldk$ decay fast with respect to hyperbolic cross structures. In addition, $(\widehat{G_3^d})_\boldk=0$ if at least one component of $\boldk$ is odd.
Therefore, we approximate the function $G_3^d$ by multivariate trigonometric polynomials $S_I^{ \Lambda } G_3^d := \sum_{ \boldk \in I } (\widehat{G_3^d})^\Lambda_\boldk \, \e^{ 2\pi \ii \boldk \cdot \circ } $ with frequencies supported on modified hyperbolic cross index sets
$I=H_{R,\text{even}}^d$ as defined in~\eqref{equ:def:H_R_even}.
We compute the Fourier coefficients $(\widehat{G_3^d})^\Lambda_\boldk$ based on samples of~$G_3^d$ and determine the relative $L_2(\T^d)$ sampling errors
$\Vert G_3^d - S_I^{ \Lambda } G_3^d\Vert_{L_2(\T^d)} / \Vert G_3^d\Vert_{L_2(\T^d)}$,
where 
\begin{equation*}
\Vert G_3^d - S_I^{ \Lambda } G_3^d\Vert_{L_2(\T^d)}
=
\sqrt{
	\Vert G_3^d\Vert_{L_2(\T^d)}^2 - \sum_{\boldk\in I} \left|\left(\widehat{G_3^d}\right)_\boldk\right|^2
	+
	\sum_{\boldk\in I} \left|\left(\widehat{G_3^d}\right)_\boldk - \left(\widehat{G_3^d}\right)^\Lambda_\boldk\right|^2
}.
\end{equation*}
We compare the numerical results from \cite[Figure~4.3b]{KaVo19}, where reconstructing single rank-1 lattices and reconstructing random multiple rank-1 lattices were used, with new results using deterministic multiple rank-1 lattices returned by \cref{alg:main1}.

As input for \cref{alg:main1}, we use reconstructing single rank-1 lattices $\Lambda(\boldz,M,I)$ with generating vectors chosen according to~\eqref{eq:lat1}.
Instead of computing the Fourier coefficients~$(\widehat{G_3^d})^\Lambda_\boldk$ of the multivariate trigonometric polynomial~$S_I^{ \Lambda } G_3^d$ by~\eqref{equ:MultRank1Cubature}, we use \cite[Algorithm~2]{KaPoVo17}, which averages over all single rank-1 lattices $\Lambda(\boldz,\tilde P_\ell)$ that are able to reconstruct a Fourier coefficient~$\widehat{f_\boldk}$ of any multivariate trigonometric polynomial~$f$ as defined in~\eqref{equ:trigPolyf} for a given frequency~$\boldk\in I$, whereas~\eqref{equ:MultRank1Cubature} uses only one single rank\mbox{-}1 lattice $\Lambda(\boldz,\tilde P_{\nu(\boldk)})$.
Note that both computation methods are based on the same samples of $G_3^d$ along the obtained deterministic multiple rank-1 lattices. The resulting relative $L_2(\T^d)$ sampling errors are visualized for spatial dimensions $d\in\{2,3,5,8\}$ in \cref{fig:G_3:rel_sampl_err_L2} as solid lines and filled markers. We observe that the errors decrease rapidly for increasing expansion parameters~$R$ of the hyperbolic cross $I=H_{R,\text{even}}^d$ and correspondingly increasing number of samples. In addition, we consider reconstructing single rank-1 lattices generated by \cite[Algorithm~3.7]{kaemmererdiss} as input for \cref{alg:main1} and obtain results which are very close and therefore omit their plots.

\begin{figure}[!ht]
\begin{tikzpicture} \begin{loglogaxis}[enlargelimits=false,xmin=1,xmax=1e9,ymin=5e-11,ymax=2e0,ytick={1e-10,1e-8,1e-6,1e-4,1e-2,1},height=0.45\textwidth, width=0.95\textwidth, grid=major, xlabel={number of samples}, ylabel={$\Vert G_3^d - S_I^{ \Lambda } G_3^d\Vert_{L_2(\T^d)} / \Vert G_3^d\Vert_{L_2(\T^d)}$}, legend style={at={(0.35,1.02)}, anchor=south,legend columns=4,legend cell align=left, font=\footnotesize, },
]
\addplot[forget plot,dashed,mark options={solid},blue,mark=o,mark size=2.5] coordinates {
(5,6.973e-02) (13,1.939e-02) (41,1.854e-03) (145,2.324e-04) (545,2.846e-05) (2113,2.789e-06) (8321,2.796e-07) (33025,2.625e-08) (131585,2.437e-09) (525313,2.226e-10)
};
\addplot[forget plot,dotted,very thick,mark options={solid},blue,mark=*,mark size=2.25] coordinates {
	(11,6.759e-02) (29,1.386e-02) (59,1.688e-03) (131,2.390e-04) (1221,2.929e-05) (1991,2.856e-06) (7407,2.912e-07) (23127,2.758e-08) (43347,2.519e-09) (95289,2.374e-10) };
\addplot[blue,mark=*,mark size=2.5] coordinates {
(7,8.314e-02) (53,1.389e-02) (99,1.910e-03) (215,2.464e-04) (801,3.450e-05) (2071,2.998e-06) (4929,2.995e-07) (15495,2.937e-08) (40725,2.739e-09) (84909,2.523e-10) };
\addlegendentry{$d$=2}
\addplot[forget plot,dashed,mark options={solid},red,mark=square,mark size=2] coordinates {
(7,1.091e-01) (29,3.408e-02) (97,4.736e-03) (395,6.420e-04) (1721,9.394e-05) (5161,1.271e-05) (21569,1.514e-06) (85405,1.527e-07) (359213,1.525e-08) (1383595,1.464e-09) (5416219,1.385e-10)
};
\addplot[forget plot,dotted,very thick,mark options={solid},red,mark=square*,mark size=1.75] coordinates {
	(29,1.387e-01) (53,2.273e-02) (423,4.627e-03) (1827,6.747e-04) (2649,9.872e-05) (15557,1.367e-05) (32411,1.604e-06) (131163,1.577e-07) (344841,1.574e-08) (666907,1.533e-09) (1559059,1.448e-10)
};
\addplot[red,mark=square*,mark size=2] coordinates {
(7,1.091e-01) (29,3.408e-02) (279,4.894e-03) (1225,7.080e-04) (3811,1.035e-04) (9255,1.443e-05) (27637,1.682e-06) (73385,1.679e-07) (207443,1.673e-08) (529149,1.578e-09) (1256205,1.490e-10) };	
\addlegendentry{$d$=3}
\addplot[forget plot,dashed,mark options={solid},darkgreen,mark=diamond,mark size=2.5] coordinates {
(11,2.538e-01) (81,6.102e-02) (543,1.235e-02) (3079,2.513e-03) (14253,4.538e-04) (78167,8.217e-05) (404035,1.319e-05) (2328905,2.000e-06) (12181705,2.776e-07) (70968649,4.470e-08)
};
\addplot[forget plot,dotted,very thick,mark options={solid},darkgreen,mark=diamond*,mark size=2.25] coordinates {
 (81,1.793e-01) (679,5.123e-02) (1977,1.279e-02) (11291,2.512e-03) (39225,4.781e-04) (141919,8.373e-05) (401867,1.362e-05) (1229741,2.083e-06) (3665943,2.874e-07) (9836645,4.470e-08)
};
\addplot[darkgreen,mark=diamond*,mark size=2.5] coordinates {
(11,2.538e-01) (275,5.687e-02) (1877,1.435e-02) (6671,2.934e-03) (22383,5.170e-04) (80365,9.121e-05) (269495,1.447e-05) (823821,2.154e-06) (2600909,2.967e-07) (7529071,4.712e-08) };
\addlegendentry{$d$=5}
\addplot[forget plot,dashed,mark options={solid},brown,mark=hexagon,mark size=2.25] coordinates {
(17,6.810e-01) (255,1.087e-01) (2895,3.003e-02) (23375,7.830e-03) (184859,1.909e-03) (1248979,4.331e-04) (11051805,9.125e-05) (84391053,1.814e-05) (600266399,3.389e-06)
};
\addplot[forget plot,dotted,very thick,mark options={solid},brown,mark=hexagon*,mark size=2] coordinates {
	(37,3.336e-01) (1221,1.016e-01) (13763,2.974e-02) (63159,8.251e-03) (283521,1.976e-03) (1253697,4.470e-04) (5054425,9.463e-05) (17857181,1.868e-05) (65099159,3.485e-06) };
\addplot[brown,mark=hexagon*,mark size=2.25] coordinates {
(17,6.810e-01) (811,1.106e-01) (7089,3.385e-02) (44771,9.289e-03) (197643,2.094e-03) (898975,4.653e-04) (3552845,9.794e-05) (12583473,1.920e-05) (43452535,3.572e-06) };
\addlegendentry{$d$=8}
\end{loglogaxis}
\end{tikzpicture}
\caption{Relative $L_2(\T^d)$ sampling errors
for $G_3^d$ with respect to the number of samples for reconstructing single rank-1 lattices (dashed lines, unfilled markers), reconstructing random multiple rank-1 lattices (dotted lines, filled markers), and reconstructing deterministic multiple rank-1 lattices (solid lines, filled markers),
when using the frequency index sets $I:=H_{R,\mathrm{even}}^d$. Results for single rank-1 lattices from \cite[Figure~2.14]{volkmerdiss} and for reconstructing random multiple rank-1 lattices from \cite[Figure~4.3]{KaVo19}.
}
\label{fig:G_3:rel_sampl_err_L2}
\end{figure}
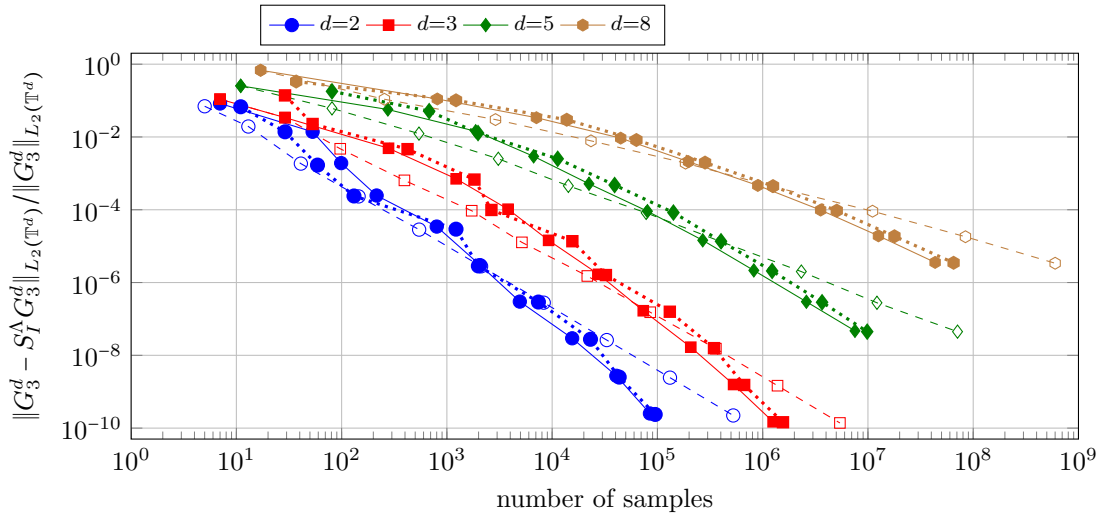

Moreover, the relative errors from \cite[Figure~4.3b]{KaVo19} when using reconstructing random multiple rank-1 lattices are shown in \cref{fig:G_3:rel_sampl_err_L2} as dotted lines and filled markers. We observe that the obtained number of samples and errors are similar to the deterministic ones. The results for the deterministic multiple rank-1 lattice seem to be slightly better for $d\in\{3,5,8\}$. In addition, the relative errors from \cite[Figure~4.3b]{KaVo19} when directly sampling along reconstructing single rank-1 lattices are drawn as dashed lines and unfilled markers. It has already been observed in~\cite{KaVo19} that in the beginning for smaller expansion parameters~$R$ and consequently smaller number of samples, the single rank-1 lattices perform better until a certain expansion parameter~$R$ has been reached. Afterwards, the multiple rank-1 lattices clearly outperform the single ones.

\FloatBarrier

\subsection{Deterministic multiple rank-1 lattices with decreasing lattice size for reconstruction of trigonometric polynomials}

Besides generating deterministic multiple rank-1 lattices according to \cref{thm:not_more_than_half_of_frequencies_collide,alg:main1}, we have also discussed the alternate approach of \cref{theorem:lattice_reduction}, where the theoretical results for function approximation, as mentioned in \cref{rem:ApproximationResult}, cannot be applied directly, but the number of required samples for the reconstruction of multivariate trigonometric polynomials may be distinctly smaller.

We start with symmetric hyperbolic cross type index sets $I=H_{R,\text{even}}^d$ and apply the generation strategy of \cref{theorem:lattice_reduction} on reconstructing single rank-1 lattices $\Lambda(\boldz,M,I)$ generated by \cite[Algorithm~3.7]{kaemmererdiss}. We visualize the resulting oversampling factors \#samples$\,/\,|I| = (1-L+\sum_{\ell=0}^{L-1} \tilde{P}_\ell)/|I|$ in \cref{fig:samples:symhceven_reduction:ratio:cbc} for spatial dimensions $d\in\{2,3,\ldots,9\}$ and various expansion parameters~$R$. For the considered test cases, we observe that the oversampling factors are well below 3. When starting with single rank-1 lattices according to~\eqref{eq:lat1}, the observed oversampling factors only differ slightly, cf.\ \cref{fig:samples:symhceven_reduction:ratio:lat1}.

\begin{figure}[!ht]
	\begin{center}
	\begin{tikzpicture}
	\begin{axis}[font=\footnotesize,hide axis,xmin=0,xmax=1,ymin=0,ymax=1,width=\textwidth,legend style={legend cell align=left},legend columns=8]
	\addlegendimage{black,mark=square,mark size=3pt,mark options={solid}}
	\addlegendentry{$d\!=\!2$}
	\addlegendimage{blue,mark=triangle,mark size=3pt,mark options={solid}}
	\addlegendentry{$d\!=\!3$}
	\addlegendimage{red,mark=o,mark size=3pt,mark options={solid}}
	\addlegendentry{$d\!=\!4$}
	\addlegendimage{magenta,mark=pentagon,mark size=3pt,mark options={solid}}
	\addlegendentry{$d\!=\!5$}
	\addlegendimage{brown,mark=triangle,mark size=3pt,mark options={solid,rotate=180}}
	\addlegendentry{$d\!=\!6$}
	\addlegendimage{darkgreen,mark=pentagon,mark size=3pt,mark options={solid,rotate=180}}
	\addlegendentry{$d\!=\!7$}
	\addlegendimage{cyan,mark=star,mark size=3pt,mark options={solid}}
	\addlegendentry{$d\!=\!8$}
	\addlegendimage{black,mark=diamond,mark size=3pt,mark options={solid}}
	\addlegendentry{$d\!=\!9$}
	\end{axis}
	\end{tikzpicture}
\end{center}
\vspace{-0.5em}
\noindent
\subfloat[{$\Lambda(\boldz,M,I)$ generated by \cite[Algorithm~3.7]{kaemmererdiss} for symmetric hyperbolic cross index sets $I=H_{R,\text{even}}^d$}]{\label{fig:samples:symhceven_reduction:ratio:cbc}
	\begin{tikzpicture}[baseline]
	\begin{axis}[
	font=\footnotesize,
	enlarge x limits=true,
	enlarge y limits=true,
	height=0.33\textwidth,
	grid=major,
	width=0.49\textwidth,
	ymin=1,ymax=3.3,
	xmode=log,
	xlabel={$|I|$},
	ylabel={\#samples$/|I|$},
	legend style={legend cell align=left, at={(1.0,1.38)}},
legend columns = 4,
	]
	\addplot[black,mark=square,mark size=3pt,mark options={solid}] coordinates {
(1,2.000e+00) (5,1.000e+00) (13,1.000e+00) (29,1.207e+00) (65,1.492e+00) (145,1.593e+00) (329,2.021e+00) (733,1.955e+00) (1633,1.936e+00) (3605,2.045e+00) (7913,2.050e+00) (17217,2.152e+00)
	};
\addplot[blue,mark=triangle,mark size=3pt,mark options={solid}] coordinates {
(1,2.000e+00) (7,1.000e+00) (25,1.160e+00) (69,1.841e+00) (177,1.633e+00) (441,1.952e+00) (1097,1.892e+00) (2693,2.162e+00) (6529,2.118e+00) (15645,2.187e+00) (37025,2.254e+00) (86593,2.294e+00)
	};
\addplot[red,mark=o,mark size=3pt,mark options={solid}] coordinates {
(1,2.000e+00) (9,1.222e+00) (41,1.634e+00) (137,1.934e+00) (401,2.062e+00) (1105,2.108e+00) (2977,2.133e+00) (7897,2.185e+00) (20609,2.257e+00) (52953,2.348e+00) (133905,2.338e+00) (333457,2.407e+00)
	};
\addplot[magenta,mark=pentagon,mark size=3pt,mark options={solid}] coordinates {
(1,2.000e+00) (11,1.000e+00) (61,1.557e+00) (241,2.095e+00) (801,2.086e+00) (2433,2.169e+00) (7073,2.219e+00) (20073,2.290e+00) (55873,2.284e+00) (152713,2.323e+00) (409825,2.437e+00)
	};
\addplot[brown,mark=triangle,mark size=3pt,mark options={solid,rotate=180}] coordinates {
(1,2.000e+00) (13,1.000e+00) (85,1.659e+00) (389,2.064e+00) (1457,2.200e+00) (4865,2.191e+00) (15241,2.192e+00) (46069,2.325e+00) (135905,2.403e+00) (392717,2.424e+00)
	};
\addplot[darkgreen,mark=pentagon,mark size=3pt,mark options={solid,rotate=180}] coordinates {
(1,2.000e+00) (15,1.133e+00) (113,1.779e+00) (589,2.097e+00) (2465,2.163e+00) (9017,2.267e+00) (30409,2.396e+00) (97709,2.389e+00) (304321,2.391e+00) (925445,2.476e+00)
	};
\addplot[cyan,mark=star,mark size=3pt,mark options={solid}] coordinates {
(1,2.000e+00) (17,1.000e+00) (145,1.731e+00) (849,2.069e+00) (3937,2.284e+00) (15713,2.337e+00) (56961,2.391e+00) (194353,2.425e+00) (637697,2.471e+00) (2034289,2.526e+00)
	};
\addplot[black,mark=diamond,mark size=3pt,mark options={solid}] coordinates {
(1,2.000e+00) (19,1.000e+00) (181,2.017e+00) (1177,2.105e+00) (6001,2.295e+00) (26017,2.328e+00) (101185,2.380e+00) (366289,2.454e+00) (1264513,2.522e+00)
	};
\end{axis}
	\end{tikzpicture}
}
\hfill
\subfloat[$\boldz$ and $M$ according to~\eqref{eq:lat1} for symmetric hyperbolic cross index sets $I=H_{R,\text{even}}^d$]{\label{fig:samples:symhceven_reduction:ratio:lat1}
	\begin{tikzpicture}[baseline]
	\begin{axis}[
	font=\footnotesize,
	enlarge x limits=true,
	enlarge y limits=true,
	height=0.33\textwidth,
	grid=major,
	width=0.49\textwidth,
	ymin=1,ymax=3.3,
	xmode=log,
	xlabel={$|I|$},
	ylabel={\#samples$/|I|$},
	legend style={legend cell align=left, at={(1.0,1.38)}},
legend columns = 4,
	]
	\addplot[black,mark=square,mark size=3pt,mark options={solid}] coordinates {
(1,2.000e+00) (5,1.400e+00) (13,1.462e+00) (29,1.483e+00) (65,2.354e+00) (145,1.759e+00) (329,1.845e+00) (733,1.742e+00) (1633,1.976e+00) (3605,1.963e+00) (7913,2.105e+00) (17217,2.134e+00) (37241,2.117e+00)
	};
\addplot[blue,mark=triangle,mark size=3pt,mark options={solid}] coordinates {
(1,2.000e+00) (7,1.000e+00) (25,1.160e+00) (69,1.870e+00) (177,1.915e+00) (441,2.116e+00) (1097,2.136e+00) (2693,2.080e+00) (6529,2.173e+00) (15645,2.195e+00) (37025,2.220e+00) (86593,2.287e+00) (200225,2.326e+00)
	};
\addplot[red,mark=o,mark size=3pt,mark options={solid}] coordinates {
(1,2.000e+00) (9,1.222e+00) (41,1.683e+00) (137,1.467e+00) (401,1.783e+00) (1105,1.981e+00) (2977,2.135e+00) (7897,2.194e+00) (20609,2.222e+00) (52953,2.296e+00) (133905,2.372e+00) (333457,2.363e+00) (818449,2.409e+00)
	};
\addplot[magenta,mark=pentagon,mark size=3pt,mark options={solid}] coordinates {
(1,2.000e+00) (11,1.000e+00) (61,1.623e+00) (241,1.963e+00) (801,2.144e+00) (2433,2.097e+00) (7073,2.201e+00) (20073,2.290e+00) (55873,2.316e+00) (152713,2.363e+00) (409825,2.381e+00)
	};
\addplot[brown,mark=triangle,mark size=3pt,mark options={solid,rotate=180}] coordinates {
(1,2.000e+00) (13,1.308e+00) (85,1.871e+00) (389,1.977e+00) (1457,2.123e+00) (4865,2.208e+00) (15241,2.218e+00) (46069,2.280e+00) (135905,2.349e+00) (392717,2.409e+00) (1112313,2.456e+00)
	};
\addplot[darkgreen,mark=pentagon,mark size=3pt,mark options={solid,rotate=180}] coordinates {
(1,2.000e+00) (15,1.133e+00) (113,2.097e+00) (589,2.022e+00) (2465,2.083e+00) (9017,2.211e+00) (30409,2.303e+00) (97709,2.373e+00) (304321,2.396e+00) (925445,2.484e+00)
	};
\addplot[cyan,mark=star,mark size=3pt,mark options={solid}] coordinates {
(1,2.000e+00) (17,1.000e+00) (145,1.566e+00) (849,2.105e+00) (3937,2.243e+00) (15713,2.255e+00) (56961,2.343e+00) (194353,2.393e+00) (637697,2.447e+00) (2034289,2.483e+00)
	};
\addplot[black,mark=diamond,mark size=3pt,mark options={solid}] coordinates {
(1,2.000e+00) (19,1.000e+00) (181,1.972e+00) (1177,2.059e+00) (6001,2.246e+00) (26017,2.321e+00) (101185,2.376e+00) (366289,2.419e+00) (1264513,2.482e+00)
	};
\end{axis}
	\end{tikzpicture}
}
\\[1em]
	\begin{center}
	\begin{tikzpicture}
	\begin{axis}[font=\footnotesize,hide axis,xmin=0,xmax=1,ymin=0,ymax=1,width=\textwidth,legend style={legend cell align=left},legend columns=5]
	\addlegendimage{black,mark=square,mark size=3pt,mark options={solid}}
	\addlegendentry{$d\!=\!2$}
	\addlegendimage{blue,mark=triangle,mark size=3pt,mark options={solid}}
	\addlegendentry{$d\!=\!3$}
	\addlegendimage{red,mark=o,mark size=3pt,mark options={solid}}
	\addlegendentry{$d\!=\!4$}
	\addlegendimage{magenta,mark=pentagon,mark size=3pt,mark options={solid}}
	\addlegendentry{$d\!=\!6$}
	\addlegendimage{brown,mark=triangle,mark size=3pt,mark options={solid,rotate=180}}
	\addlegendentry{$d\!=\!10$}
	\addlegendimage{darkgreen,mark=pentagon,mark size=3pt,mark options={solid,rotate=180}}
	\addlegendentry{$d\!=\!100$}
	\addlegendimage{cyan,mark=star,mark size=3pt,mark options={solid}}
	\addlegendentry{$d\!=\!1000$}
	\addlegendimage{black,mark=diamond,mark size=3pt,mark options={solid}}
	\addlegendentry{$d\!=\!10000$}
	\end{axis}
	\end{tikzpicture}
\end{center}
\vspace{-0.5em}
\noindent
\subfloat[{$\Lambda(\boldz,M,I)$ generated by \cite[Algorithm~3.7]{kaemmererdiss} for random frequency sets $I\subset\{-64,-63,\ldots,64\}^d$}]{\label{fig:samples:rand_reduction:ratio:cbc}
	\begin{tikzpicture}[baseline]
	\begin{axis}[
	font=\footnotesize,
	enlarge x limits=true,
	enlarge y limits=true,
	height=0.33\textwidth,
	grid=major,
	width=0.49\textwidth,
	ymin=1,ymax=3.3,
	xmode=log,
	xlabel={$|I|$},
	ylabel={\#samples$/|I|$},
legend style={legend cell align=left, at={(1.0,1.53)}},
	legend columns = 3,
	]
	\addplot[black,mark=square,mark size=3pt,mark options={solid}] coordinates {
(10,2.100e+00) (100,2.610e+00) (1000,2.672e+00) (10000,2.130e+00)
	};
\addplot[blue,mark=triangle,mark size=3pt,mark options={solid}] coordinates {
(10,2.300e+00) (100,2.470e+00) (1000,2.683e+00) (10000,2.785e+00) (100000,2.769e+00)
	};
\addplot[red,mark=o,mark size=3pt,mark options={solid}] coordinates {
(10,2.300e+00) (100,2.560e+00) (1000,2.725e+00) (10000,2.789e+00) (100000,2.848e+00)
	};
\addplot[magenta,mark=pentagon,mark size=3pt,mark options={solid}] coordinates {
(10,2.300e+00) (100,2.580e+00) (1000,2.655e+00) (10000,2.810e+00) (100000,2.841e+00)
	};
\addplot[brown,mark=triangle,mark size=3pt,mark options={solid,rotate=180}] coordinates {
(10,2.400e+00) (100,2.580e+00) (1000,2.715e+00) (10000,2.796e+00) (100000,2.844e+00)
	};
\addplot[darkgreen,mark=pentagon,mark size=3pt,mark options={solid,rotate=180}] coordinates {
(10,2.700e+00) (100,2.560e+00) (1000,2.685e+00) (10000,2.785e+00) (100000,2.850e+00)
	};
\addplot[cyan,mark=star,mark size=3pt,mark options={solid}] coordinates {
(10,2.100e+00) (100,2.530e+00) (1000,2.678e+00) (10000,2.792e+00) (100000,2.843e+00)
	};
\addplot[black,mark=diamond,mark size=3pt,mark options={solid}] coordinates {
(10,2.400e+00) (100,2.650e+00) (1000,2.693e+00) (10000,2.799e+00)
	};
\end{axis}
	\end{tikzpicture}
}
\hfill
\subfloat[$\boldz$ and $M$ according to~\eqref{eq:lat1} for random frequency sets $I\subset\{-64,-63,\ldots,64\}^d$]{\label{fig:samples:rand_reduction:ratio:lat1}
	\begin{tikzpicture}[baseline]
	\begin{axis}[
	font=\footnotesize,
	enlarge x limits=true,
	enlarge y limits=true,
	height=0.33\textwidth,
	grid=major,
	width=0.49\textwidth,
	ymin=1,ymax=3.3,
	xmode=log,
	xlabel={$|I|$},
	ylabel={\#samples$/|I|$},
	legend style={legend cell align=left, at={(1.0,1.53)}},
legend columns = 3,
	]
	\addplot[black,mark=square,mark size=3pt,mark options={solid}] coordinates {
(10,2.400e+00) (100,2.510e+00) (1000,2.672e+00) (10000,2.130e+00)
	};
\addplot[blue,mark=triangle,mark size=3pt,mark options={solid}] coordinates {
(10,3.300e+00) (100,2.670e+00) (1000,2.663e+00) (10000,2.784e+00) (100000,2.769e+00)
	};
\addplot[red,mark=o,mark size=3pt,mark options={solid}] coordinates {
(10,2.700e+00) (100,2.740e+00) (1000,2.716e+00) (10000,2.796e+00) (100000,2.850e+00)
	};
\addplot[magenta,mark=pentagon,mark size=3pt,mark options={solid}] coordinates {
(10,2.700e+00) (100,2.470e+00) (1000,2.729e+00) (10000,2.795e+00) (100000,2.845e+00)
	};
\addplot[brown,mark=triangle,mark size=3pt,mark options={solid,rotate=180}] coordinates {
(10,2.500e+00) (100,2.620e+00) (1000,2.693e+00) (10000,2.780e+00) (100000,2.842e+00)
	};
\addplot[darkgreen,mark=pentagon,mark size=3pt,mark options={solid,rotate=180}] coordinates {
(10,2.300e+00) (100,2.710e+00) (1000,2.737e+00) (10000,2.788e+00) (100000,2.846e+00)
	};
\addplot[cyan,mark=star,mark size=3pt,mark options={solid}] coordinates {
(10,2.800e+00) (100,2.530e+00) (1000,2.738e+00) (10000,2.808e+00) (100000,2.846e+00)
	};
\addplot[black,mark=diamond,mark size=3pt,mark options={solid}] coordinates {
(10,2.300e+00) (100,2.600e+00) (1000,2.689e+00) (10000,2.797e+00)
	};
\end{axis}
	\end{tikzpicture}
}
	\caption{Oversampling factors for deterministic reconstructing multiple rank-1 lattices constructed according to \cref{theorem:lattice_reduction}.}
\end{figure}
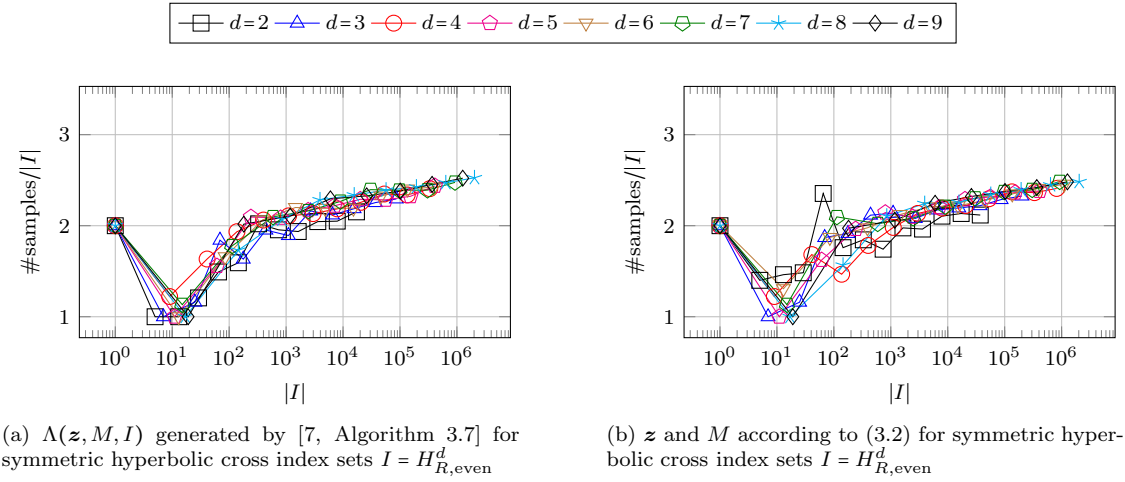
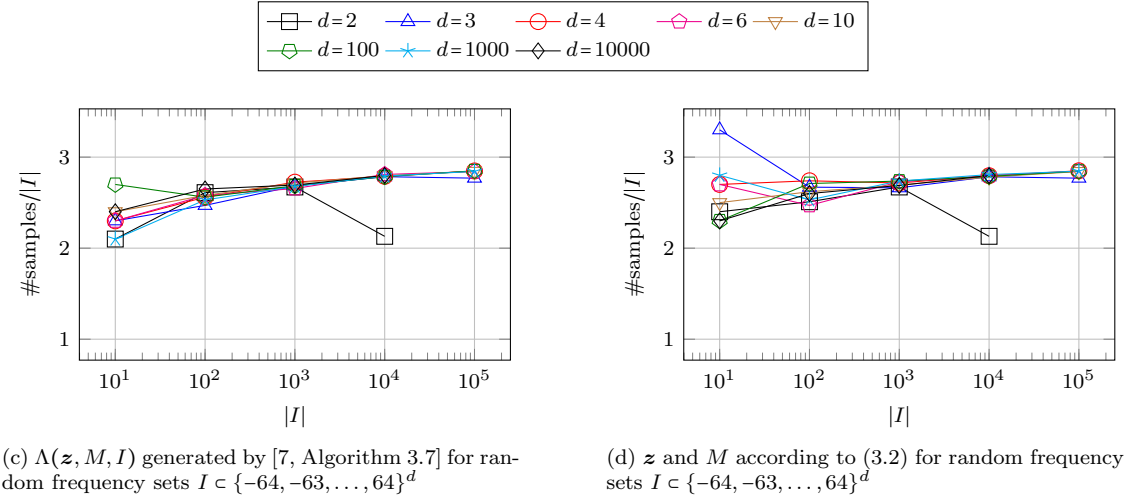

The reason for these very low oversampling factors is that during the generation process according to the proof of \cref{theorem:lattice_reduction} the prime $\tilde{P}_0$ is relatively close to $|I|$, the next prime $\tilde{P}_1$ is relatively close to $|I\setminus I_0|$, $\tilde{P}_2$ is relatively close to $|I\setminus (I_0\cup I_1)|$, and so on, where
$I_0$ contains the frequencies of $I$ which can be reconstructed by the lattice $\Lambda(\boldz,\tilde{P}_0)$ and where $I_1$ contains the frequencies of $I\setminus I_0$ which can be reconstructed by $\Lambda(\boldz,\tilde{P}_1)$. In particular, we do not have the fixed lower bound $s\leq\tilde{P}_\ell$ for all $\ell$ as in \cref{alg:main1}.

Next, we change the setting and use the frequency sets $I$ drawn uniformly randomly from cubes $[-R,R]^d\cap\Z^d$, see \cref{sec:numerics:lattices_alg1}. As before, we generate reconstructing single rank-1 lattices $\Lambda(\boldz,M, I)$ using \cite[Algorithm~3.7]{kaemmererdiss}. Then, we apply the strategy of \cref{theorem:lattice_reduction} in order to deterministically generate reconstructing multiple rank-1 lattices. We repeat the test 10 times for each setting with newly randomly chosen frequency sets~$I$ and determine the maximum number of samples over the 10 repetitions. For sparsities $|I|\in\{10,100,1\,000,10\,000\}$ in $d\in\{2,3,4,6,10,100,1\,000,10\,000\}$ spatial dimensions, we visualize the resulting oversampling factors in \cref{fig:samples:rand_reduction:ratio:cbc} for expansion parameter $R=64$ ($N_I\leq 128$). Starting with reconstructing single rank-1 lattices $\Lambda(\boldz,M, I)$ according to~\eqref{eq:lat1} as in \cref{fig:samples:rand:oversampling:lat1} changes the oversampling factors only slightly,
and the oversampling factors are still well below $4$, cf.\ \cref{fig:samples:rand_reduction:ratio:lat1}.

\section{Conclusion}
Given a known reconstructing single rank\mbox{-}1 lattice $ \Lambda(\boldz, M, I) $ for some $ d $-dimensional frequency set $ I \subset \Z^d $, we have presented two methods for deterministically transforming this lattice into smaller multiple rank\mbox{-}1 lattices.
By sampling on these lattices, any exact trigonometric polynomial with Fourier support in $ I $ can be exactly recovered by performing $ \OO{\log|I|} $ FFTs each of size bounded linearly in $ |I| $ and logarithmically in $ M \max_{ \boldk \in I } \norm{ \boldk }_1 $ up to some $\log\log$ terms. Slightly restricting the frequency set $I$ to being a subset of the $d$-dimensional $\ell_1$-ball of radius~$|I|$ and choosing suitable CBC constructed reconstructing rank-1 lattices $ \Lambda(\boldz, M, I) $ as input for Algorithm~\ref{alg:main1} actually leads to FFT sizes bounded in $\OO{|I|\log|I|}$.

Numerically, we have demonstrated that the deterministically generated multiple rank\mbox{-}1 lattices stemming from three different original reconstructing single rank\mbox{-}1 lattice constructions are competitive with single rank\mbox{-}1 lattices as well as randomly generated multiple rank\mbox{-}1 lattices in terms of number of samples used and accuracy of approximations.
These same results have additionally shown some gaps between the lattices generated in practice and the theoretical bounds, most notably when considering the oversampling factors, where for both structured and random frequency sets, these values were shown to grow only at most logarithmically in $ |I| $.

It is worth noting that though the construction of the multiple lattices happens sequentially (in that each successive choice of lattice size depends on the previous choices) the generating vector $ \boldz $ is kept constant and not adapted to previously chosen lattices.
Rather than remaining tied to the original reconstructing single rank\mbox{-}1 lattice, related work could take the path of generating new single rank\mbox{-}1 lattices which better serve the frequencies waiting to be handled after each step, so long as the potential gains in an adapted lattice scheme are carefully balanced with the associated computational cost.
\vspace*{-.75em}

\bibliographystyle{abbrv}
\bibliography{references_dmr1l_arXiv}

\end{document}